\documentclass[11pt]{amsart}
\usepackage{mathrsfs,amsthm,amsmath,amssymb,bbm,pdfsync,prettyref,hyperref,graphicx,fullpage,esint,subfig}

\usepackage{wasysym}

\theoremstyle{plain}
\newtheorem{theorem}{\protect\theoremname}[section]
\newtheorem*{theorem*}{\protect\theoremname}

\newtheorem{lemma}[theorem]{\protect\lemmaname}
\newtheorem*{lemma*}{\protect\lemmaname}

\newtheorem{proposition}[theorem]{\protect\propositionname}
\newtheorem*{proposition*}{\protect\propositionname}

\newtheorem{corollary}[theorem]{\protect\corollaryname}
\newtheorem*{corollary*}{\protect\corollaryname}

\newtheorem{fact}[theorem]{\protect\factname}
\newtheorem*{fact*}{\protect\factname}
\newtheorem{conjecture}[theorem]{\protect\conjecturename}

\theoremstyle{definition}
\newtheorem{definition}[theorem]{\protect\definitionname}
\newtheorem*{definition*}{\protect\definitionname}
\newtheorem{example}[theorem]{\protect\examplename}

\newtheorem{question}[theorem]{\protect\questionname}
\theoremstyle{remark}
\newtheorem{remark}[theorem]{\protect\remarkname}

\numberwithin{equation}{subsection}

\newcommand{\cA}{\mathcal{A}}

\newcommand{\cE}{\mathcal{E}}
\newcommand{\cF}{\mathcal{F}}

\newcommand{\fG}{\mathfrak{G}}

\newcommand{\cL}{\mathcal{L}}
\newcommand{\cM}{\mathscr{M}}
\newcommand{\cN}{\mathcal{N}}

\newcommand{\cP}{\mathcal{P}}

\newcommand{\R}{\mathbb{R}}

\newcommand{\E}{\mathbb{E}}
\newcommand{\eps}{\epsilon}

\newcommand{\tensor}{\otimes}

\newcommand{\gibbs}[1]{\left\langle#1\right\rangle}
\newcommand{\innerp}[1]{\left\langle#1\right\rangle}
\newcommand{\brac}[1]{\left\langle#1\right\rangle}

\newcommand{\sech}{\text{sech}}
\newcommand{\bm}{\mathbf{m}}
\newcommand{\abs}[1]{\lvert#1\rvert}
\newcommand{\norm}[1]{\lvert\lvert#1\rvert\rvert}

\newcommand{\Tbull}{T_\bullet}

\providecommand{\corollaryname}{Corollary}
\providecommand{\definitionname}{Definition}
\providecommand{\factname}{Fact}
\providecommand{\lemmaname}{Lemma}
\providecommand{\propositionname}{Proposition}
\providecommand{\theoremname}{Theorem}
\providecommand{\remarkname}{Remark}
\providecommand{\conjecturename}{Conjecture}
\providecommand{\examplename}{Example}
\providecommand{\claimname}{Claim}
\providecommand{\problemname}{Problem}
\providecommand{\questionname}{Question}

\newrefformat{prop}{Proposition \ref{#1}}
\newrefformat{subsub}{Section \ref{#1}}
\newrefformat{fact}{Fact \ref{#1}}
\newrefformat{cor}{Corollary \ref{#1}}
\newrefformat{exam}{Example \ref{#1}}
\newrefformat{rem}{Remark \ref{#1}}
\newrefformat{defn}{Definition \ref{#1}}
\newrefformat{conj}{Conjecture \ref{#1}}
\newrefformat{quest}{Question \ref{#1}}

\begin{document}
\title{Some properties of the phase diagram for mixed $p$-spin glasses}

\author{Aukosh Jagannath}
\address[Aukosh Jagannath]{Courant Institute of Mathematical Sciences, 251 Mercer St.\ NY, NY, USA, 10012}
\email{aukosh@cims.nyu.edu}

\author{Ian Tobasco}
\address[Ian Tobasco]{Courant Institute of Mathematical Sciences, 251 Mercer St.\ NY, NY, USA, 10012}
\email{tobasco@cims.nyu.edu}

\thanks{The final publication is available at Springer via http://dx.doi.org/10.1007/s00440-015-0691-z}

\date{\today}

\begin{abstract}
In this paper we study the Parisi variational problem for mixed $p$-spin glasses with Ising spins. 
Our starting point is a characterization of Parisi measures whose origin lies in the first order optimality 
conditions for the Parisi functional, which is known to be strictly convex. 
Using this characterization, we study the phase diagram in the temperature-external field plane. 
We begin by deriving self-consistency conditions for Parisi measures that generalize those of de Almeida and Thouless 
to all levels of Replica Symmetry Breaking (RSB) and all models. As a consequence, we conjecture that for all models 
the Replica Symmetric (RS) phase is the region determined by the natural analogue of the de Almeida-Thouless condition.
 We show that for all models, the complement of this region is in the RSB phase. 
 Furthermore, we show that the conjectured phase boundary is exactly the phase boundary in the plane less a bounded set.
 In the case of the Sherrington-Kirkpatrick model, 
 we extend this last result to show that this bounded set does not contain the critical point at zero external field.  
\end{abstract}

\maketitle

\section{Introduction}
In this paper we consider the Parisi functional, which is defined as follows. Let $\xi_0(t)$, called the model, be $\xi_0(t) = \sum_{p\geq2}\beta_{p}^{2}t^{p}$ and let $\xi(t)=\beta^2\xi_0(t)$. The Parisi functional is given by
\begin{equation}\label{eq:pfunc}
\cP(\mu;\xi_0,\beta,h)=u_{\mu}(0,h)-\frac{1}{2}\int_{0}^{1}\xi^{\prime\prime}(s)\mu[0,s]sds
\end{equation}
where $\mu\in\Pr([0,1])$ is a probability measure on the unit interval and $u_{\mu}$ is the unique weak solution of the Parisi PDE:
\begin{equation}\label{eq:PPDE}
\begin{cases}
\partial_{t}u_{\mu}(t,x)+\frac{\xi^{''}\left(t\right)}{2}\left(\partial_{xx}u_{\mu}(t,x)+\mu\left[0,t\right]\left(\partial_{x}u_{\mu}(t,x)\right)^{2}\right)=0 & (t,x)\in(0,1)\times\R\\
u_{\mu}(1,x)=\log\cosh(x)
\end{cases}
\end{equation}
Here $\beta$ is a positive real number and $h$ is a non-negative real number, and $\beta$ and $h$ are called the inverse temperature and external field respectively. We assume that there is a positive real number $\epsilon$ such that $\xi_0(1+\epsilon)<\infty$. 
The solution $u_\mu$ can be shown \cite{JagTobSC15} to be continuous in space and time
and is unique in the class of weak solutions with essentially bounded weak derivative. The basic properties of the solution of this
PDE are summarized in Sect.\ \ref{sec:appendixPDE}. 

The study of the Parisi functional is important to the study of mean-field spin glasses. 
The strict convexity of this functional was conjectured by Panchenko \cite{PanchConv05} and proven by Auffinger and Chen in \cite{AuffChen14}. 
(For an alternative proof see \cite{JagTobSC15}.) Other properties of this functional were studied in the mathematics literature
by Talagrand in \cite{TalPM06,TalPF,TalBK11vol2}, Auffinger and Chen in \cite{AuffChen14,AuffChen13}, and
Chen in \cite{Chen15}.

A question of particular significance is the nature of the minimizer of the variational formula 
\begin{equation}\label{eq:parisi-formula}
F(\xi_0,\beta,h)=\min_{\mu\in\Pr([0,1])}\cP(\mu;\xi_0,\beta,h)
\end{equation}
as $\beta$ and $h$ vary. The region in the $(\beta,h)$ plane where this measure 
consists of one atom is called the ``Replica Symmetric'' (RS) phase; the complement of this region is called the ``Replica Symmetry Breaking'' (RSB) phase;
the region where it consists of $k+1$ atoms is called the ``$k$-step Replica Symmetry Breaking'' (kRSB) phase;  
and the region where it has either infinitely many atoms or a part that is absolutely continuous with respect to the Lebesgue measure is called the ``Full Replica Symmetry Breaking'' (FRSB) phase. 
(That the measure has
no continuous singular part was first rigorously shown in \cite{AuffChen13}.) 
The phase diagram of the Parisi functional was first explored from the variational standpoint in the mathematics literature
by Auffinger and Chen in \cite{AuffChen13} in the case $h=0$. 
The importance of these questions to the field of mean field spin 
glasses and the meaning of the above terminology is explained in more detail in Sect.\ \ref{sec:background}.

The starting point of this paper is a characterization of the minimizer of \eqref{eq:parisi-formula} which is
 based on the following elementary observations. First, the tangent space of $\Pr([0,1])$ equipped with the weak topology is naturally
included in the space of finite signed measures $\cM([0,1])$ with the same topology, so that the  derivative of a functional of the form 
\eqref{eq:pfunc} in the direction of $\sigma$ is given by 
\[
\innerp{\delta \cP_\mu,\sigma}=\innerp{G_\mu,\sigma}
\]
for some continuous bounded function $G_\mu$. Then, since $\cP$ is strictly convex, the first order optimality conditions lead to the conclusion that minimizers are characterized by
the equation
\[
\mu(\{x:G_\mu(x)=\min G_\mu(x)\})=1.
\]
This is explained in Sect.\ \ref{sec:outline}.
Using this characterization, we present self-consistency conditions for the minimizer when $\beta$ and $h$ are in a given phase.  
We then present a detailed study of the phase boundary between the RS and RSB regimes. Specifically, we present a conjecture for general models which can be seen to generalize the stability conditions obtained by de Almeida and Thouless using replica theoretic techniques \cite{AT78} in the case of the Sherrington-Kirkpatrick model ($\xi_0(t) = t^2/2$). We resolve this conjecture in a large portion of the $(\beta,h)$ plane for general models. These results are based on a quantitative study of asymptotics
of gaussian integrals as the covariances and mean diverge. This will be explained in Sect.\ \ref{sec:outline}.

\subsection{Statement of main results}
Before we state the main results we need the following technical definitions. We call the minimizer of (\ref{eq:parisi-formula})
a \emph{Parisi measure}. That this measure is unique was first proven in \cite{AuffChen14}. An alternative proof is provided in
\cite{JagTobSC15}. Define the function $G_{\mu}$ for any $\mu\in\Pr\left[0,1\right]$
and $h\geq 0$ by 
\begin{equation}\label{eq:Gmu}
G_{\mu}\left(t\right)=\int_{t}^{1}\frac{\xi''\left(s\right)}{2}\left(\E_{X_{0}=h}\left[u_{x}^{2}\left(s,X_{s}\right)\right]-s\right)\, ds,
\end{equation}
where $u$ solves the Parisi PDE (\ref{eq:PPDE}) with measure $\mu$, and $X_{s}$
solves the Auffinger-Chen SDE 
\begin{equation}\label{eq:AC-SDE}
dX_{s}=\xi''\left(s\right)\mu\left[0,s\right]u_{x}\left(s,X_{s}\right)ds+\sqrt{\xi''\left(s\right)}dW_{s}.
\end{equation} 
The properties of this SDE are summarized in Sect.\ \ref{sec:appendix}.

We begin with the following useful characterization of the optimizer. 
\begin{proposition}\label{prop:opt-conds}
$\mu$ is a Parisi measure if and only if
\begin{equation}\label{eq:EL-rel}
\mu(\{t:G_{\mu}(t)=\min G_{\mu}\})=1.
\end{equation}
Furthermore, if $\mu$ is a Parisi measure, it must satisfy the self-consistency conditions,
\begin{equation}\label{eq:self-consistency}
\begin{cases}
\E_h u_{x}^2(q,X_q)&=q\\
\xi''(q)\E_h u_{xx}^2(q,X_q)&\leq 1
\end{cases}
\end{equation}
for all $q\in \text{supp}\,\mu$.
\end{proposition}
\begin{remark}
In the case that $\mu$ is 1-atomic and $\xi$ corresponds to SK, (\ref{eq:self-consistency})
are exactly the conditions of de Almeida and Thouless \cite{AT78}. See Sect.\ \ref{sec:background} for a brief discussion.
\end{remark}
\begin{remark}
These results and others from Sect.\ \ref{sec:First-Var} were independently obtained by Chen in \cite{Chen15}. In order to make this presentation self-contained, we present an alternative approach in Sect. \ref{sec:First-Var}. 
\end{remark}
\begin{remark}
A characterization similar to (\ref{eq:EL-rel}) was obtained by Talagrand in \cite{TalSphPF06} for the related
spherical mixed $p$-spin glass model, and used by Auffinger and Chen in \cite{AuffChen13} to prove Full RSB for a subclass of such models. 
The self-consistency conditions can be derived using the work of Talagrand \cite{TalPM06,TalPF}, where he states the result in 
the case of k-atomic measures and even $p$, and 
the work of Toninelli \cite{Ton02} where he works with the SK model and 1-atomic measures.
In particular, the self-consistency conditions can be seen as a generalization of Toninelli's work 
as well as a generalization of the conditions of de Almeida and Thouless,
to the case of general models and general levels of RSB, even full. 
\end{remark}
\begin{remark}
Note that an immediate consequence of the self-consistency conditions is that if we let 
$\beta_*=\frac{1}{\sqrt{\xi''_0(1)}}$, then the region $\beta\leq \beta_*$ is in the RS phase (see Lemma \ref{lem:beta-alr}).
\end{remark}

%
%

The remainder of this subsection is regarding our results on the RS to RSB phase transition. We use the following notation throughout.
Let $Z$ denote a standard Gaussian random variable and define 
\begin{equation}\label{eq:q-alpha-def}
\begin{aligned}
Q_*(\beta,h) &= \{q\in[0,1] \,:\, \E\tanh^2( \sqrt{\xi'(q)}Z+ h)=q \}\\
\alpha(q,\beta,h) &= \,\xi''(q)\E \sech^4(\sqrt{\xi'(q)}Z+ h)\\
\alpha(\beta,h) &= \min_{q\in Q_*}\,\alpha(q,\beta,h)\\
q_*(\beta,h) &= \max\,\{ q\in Q_* \,:\, \alpha(\beta,h) = \alpha(q,\beta,h)\}.
\end{aligned}
\end{equation}
By a continuity argument one can show that $Q_*(\beta,h)$ is closed, $q_*\in Q_*(\beta,h)$, and  $\alpha(\beta,h) = \alpha(q_*,\beta,h)$.
We call the level set $\alpha(\beta,h)=1$ the \emph{generalized AT-line}. There are many subtle questions regarding these quantities,
for example, ``is the set $\alpha=1$ a curve?'', ``is the set $Q_*$ a singleton?''. 
For a short discussion regarding these questions see Sect.\ \ref{sec:ATline-line}.
We remark here that a consequence of our methods of proof is that there will be only one $q$ in $Q_*$ with $\alpha(q)=\alpha$, in the region of $(\beta,h)$ where they apply.

With this in mind, we state our conjecture regarding the RS phase. Let
\begin{equation}\label{eq:AT}
AT=\{(\beta,h):\alpha(\beta,h)\leq 1\},
\end{equation}
and 
\[
RS=\{(\beta,h):\text{the Parisi measure is 1-atomic}\}.
\] 
\begin{conjecture}\label{conj:ATline}
We have the equality
\[
AT=RS.
\]
\end{conjecture}

\begin{remark}
In \cite{AT78}, de Almeida and Thouless derived this characterization of the RS phase for the SK model  using a replica theoretic 
stability analysis. The above conjecture is a natural generalization of their characterization to all mixed $p$-spin glass models. We note, however, that 
there are arguments in the literature, e.g. \cite{Panch05}, pertaining to other spin glass models which suggest that such a conjecture may not 
be true at this level of generality. Furthermore a comparison of Figures \ref{fig:1a} and \ref{fig:2a} with known results from \cite{AuffChen13} casts doubt on this conjecture in the setting of ``non-uniformly parabolic models'', where $\xi_0''(0)=0$. Analytical and numerical evidence suggests that this conjecture should hold at least in the setting of ``uniformly parabolic models'', where $\xi''_0(0)>0$.
\end{remark}

We begin by observing the following.
\begin{theorem}\label{thm:RSBalphabigger1} For any model $\xi_0$,  $RS \subset AT$.
\end{theorem}
Our main result regarding the RS to RSB phase transition for general models is as follows.
\begin{theorem}\label{thm:long-time}
For any model $\xi_0$ and positive external field $h_0>0$, there is a $\beta_u$ such that for $\beta\geq\beta_u$ and $h\geq h_0$, the region $\alpha\leq 1$ is in the RS phase. That is, 
\[
AT\cap \{ \beta \geq \beta_u,\ h\geq h_0\} \subset RS.
\]
Furthermore, if $\xi_0''(0)> 0$, then we can take $h_0=0$.
\end{theorem}
\begin{remark} \label{rem:h-is-zero}
 We note here that our proof is quantitative. In particular, one could calculate $\beta_u$ for which
the statement holds.
\end{remark}
\begin{remark}
It is still mysterious as to why ``uniform parabolicity'', $\xi_0''(0)>0$, has such a dramatic effect on this
variational problem. For our arguments, the key difference
is in the nature of an \emph{a priori} lower bound on $q_*$, specifically Lemma \ref{lem:lbq} and the remark thereafter.
\end{remark}

Finally we have the following technical result, which is crucial to the proof of the main theorem above.
In this result, we show that a region that is sufficiently buffered away from, but parallel to the conjectured boundary at $\alpha=1$ is in the RS phase. This allows us to assume a lower bound on $\alpha$ 
and to focus our efforts on the region near the generalized AT line. Let $\Lambda_0=(\pi^2-3)/(6\sqrt{2\pi})$.
\begin{proposition}\label{prop:23-AT}
 For any model $\xi_0$, 
\begin{equation}\label{eq:23-ATline}
\{ h>0,\ \alpha \leq \frac{2}{3}\frac{\xi''_0(q_*)}{\xi''_0(1)}\left(1-\frac{\Lambda_0\xi''_0(1)}{\beta(\xi'_0(q_*))^{3/2}}\right)\} \subset RS.
\end{equation}
\end{proposition}

A relatively straightforward argument shows that for any 
$\beta_0$ there is an $h_1$ such that the region $[0,\beta_0]\times[h_1,\infty)$ is in the RS phase (See  
Sect. \ref{sec:elementaryRS}). The set in the temperature-external field plane on which the above bounds fail
is then upper bounded in temperature and  external field, and is thus bounded. Furthermore if $\xi''(0)>0$, then the temperatures
for which the arguments fail are also lower bounded. 
\subsubsection{The Sherrington-Kirkpatrick model}
As an example of our techniques, we now discuss how our results relate to the Sherrington-Kirkpatrick model. In this case, it is known \cite{Guer01,TalBK11vol2}
that at positive external field, $h>0$, or high temperature, $\beta <1$, $Q_*$ is a singleton.  The question of whether or not $\alpha=1$ is a curve
is still a difficult question. This is explained in Sect.\ \ref{sec:ATline-line}. 

In this case, recall that $\beta_u$ from Theorem \ref{thm:long-time} does not
depend on $h_0$. Also note that $\beta_*=1$ in this case. Combining the above results with Lemma \ref{lem:sk-3/2}, which concerns
 the region near the critical  point $(\beta,h)=(1,0)$, we get the following improvement.

\begin{theorem}\label{thm:SK-32}
For the SK model, there are $\beta_u,h_1>0$ such that
\[
AT \cap \{ \beta \geq \beta_u\text{ or }\beta\leq \beta_u,h \geq h_1\text{ or } \beta \leq 3/2\} \subset RS.
\] 
\end{theorem}


We note that the region in which this theorem applies does not contain the critical point  $(\beta,h)=(1,0)$.
See Figure \ref{fig:RSknown} for a schematic diagram of the optimal region for our arguments.

\begin{figure}[t]
\centering \subfloat[]{\includegraphics[width=.5\textwidth]{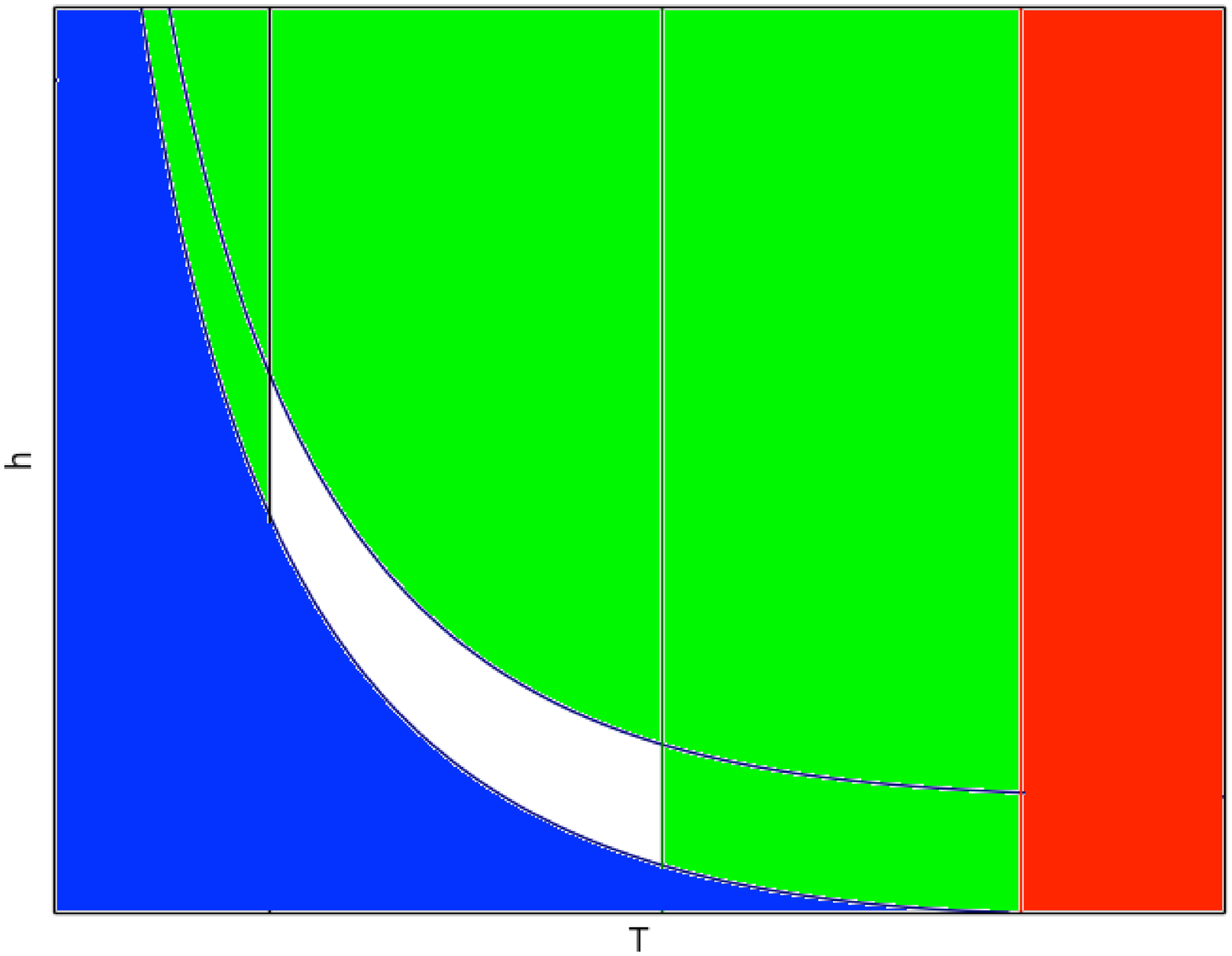}\label{fig:1a}}
 \subfloat[]{\includegraphics[width=.5\textwidth]{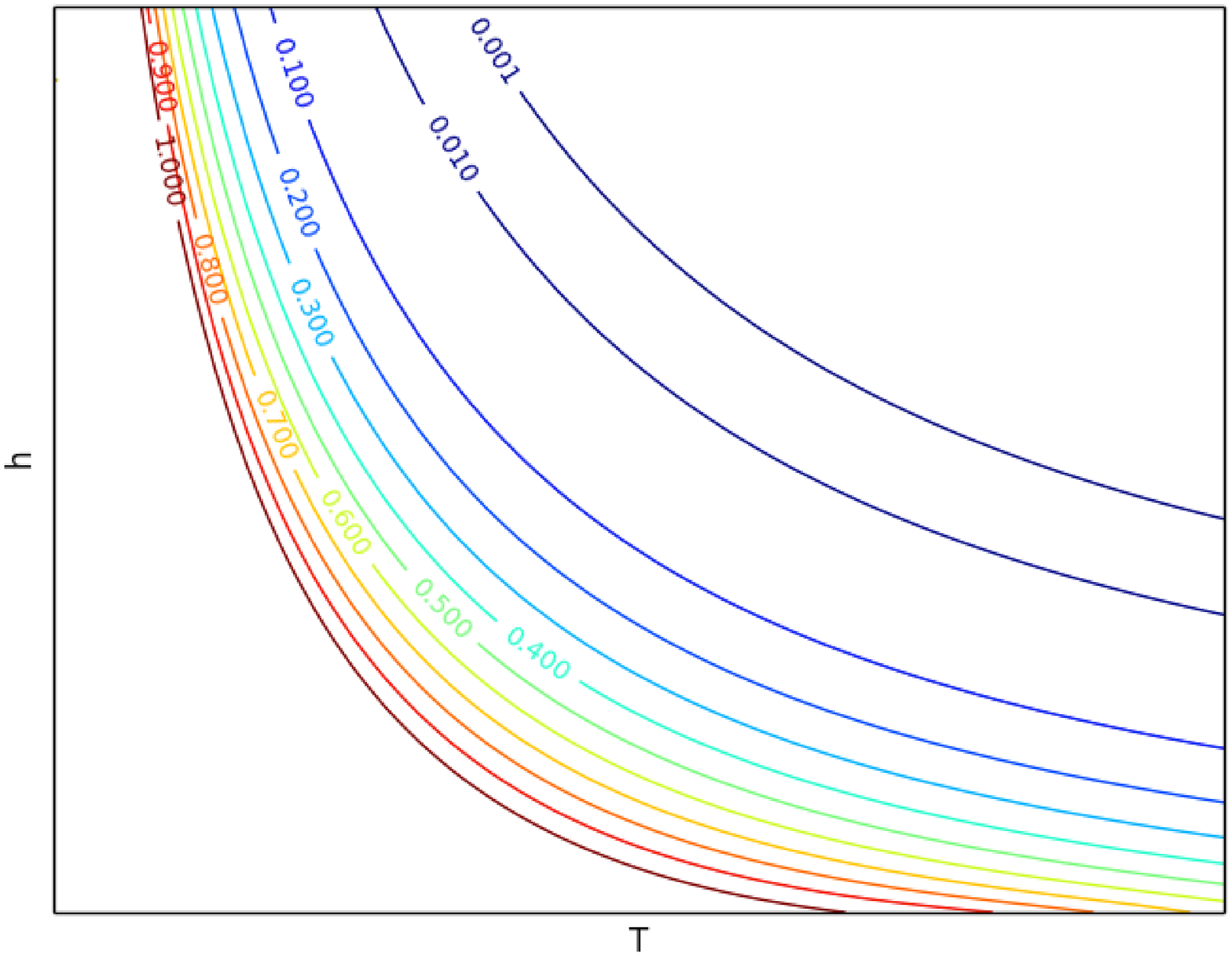}\label{fig:1b}}\\
  \subfloat[]{\includegraphics[width=.5\textwidth]{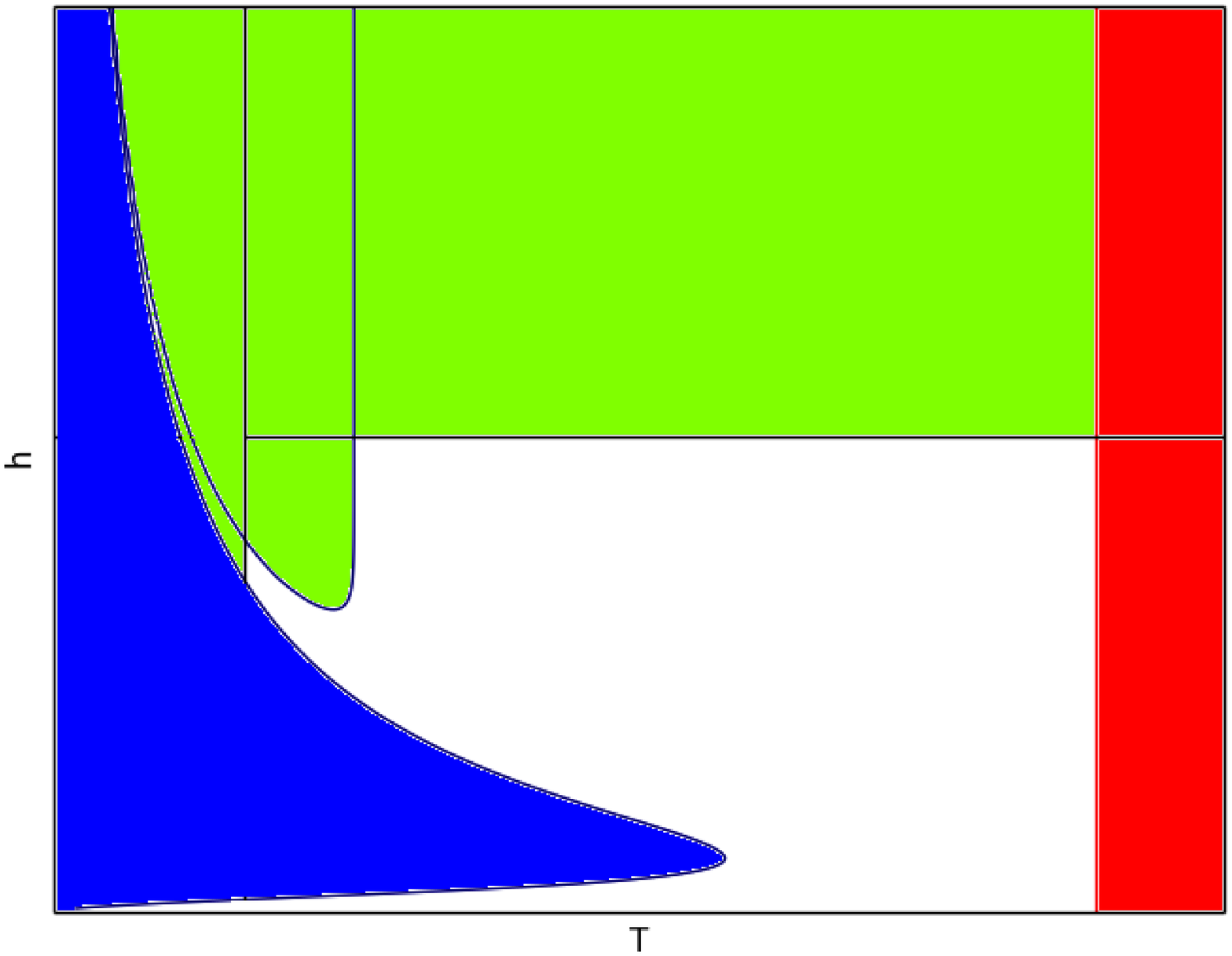}\label{fig:2a}}
 \subfloat[]{\includegraphics[width=.5\textwidth]{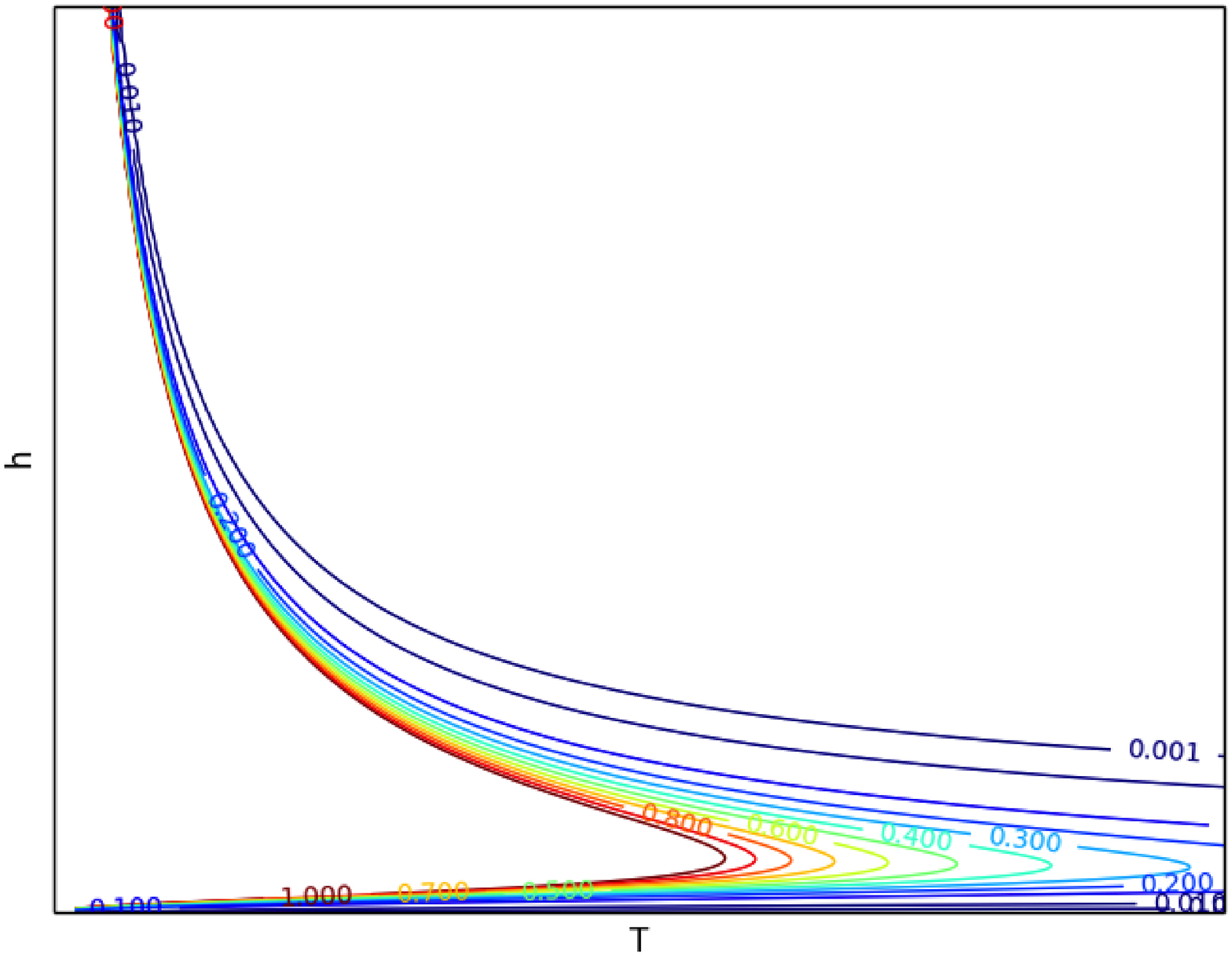}\label{fig:2b}}

 \caption{ Figure \ref{fig:1a} is a schematic of where we know RS for the SK model ($\xi_0(t) = t^2/2$) in the coordinates $(T,h)$ where
 $T=\frac{1}{\beta}$ is the temperature. The red region denotes
 the region above $1/\beta_*$. The blue region denotes the region where we know RSB.
 The green region is where we know RS. The blue curve is the curve from Proposition \ref{prop:23-AT},
 the black line (furthest to the left) is the line from Theorem \ref{thm:long-time}, and the green line (the second from the left) 
 is the line from Theorem \ref{thm:SK-32}. We note here that the aforementioned blue curve is asymptotically equivalent
 to the AT line in these coordinates. Figure \ref{fig:1b} is a plot of the level sets $\alpha=const.$\ for various values of $\alpha\leq 1$. The main idea of our analysis is to study the problem along individual level sets, and to compare with a ``limiting'' problem at $\beta,h = \infty$. Figures \ref{fig:2a}-\ref{fig:2b} are the same as the above for the pure 4-spin model $\xi_0(t)=t^4/4$.}
 \label{fig:RSknown}
\end{figure}

\subsection{Acknowledgements}
We would like to thank our advisors G.\ Ben Arous and R.V.\ Kohn  for their support. We would like to thank
anonymous referees for their very helpful suggestions regarding the exposition of this paper.
We would like to thank the New York University GRI Institute in Paris for its hospitality during the preparation of this paper.
This research was conducted while A.J.\ was supported by a National Science Foundation Graduate Research Fellowship DGE-0813964; 
and National Science Foundation grants DMS-1209165 and OISE-0730136, and while I.T.\ was supported by a National Science Foundation Graduate Research Fellowship
DGE-0813964; and National Science Foundation grants OISE-0967140 and  DMS-1311833.

\subsection{Outline of proofs of main results and discussion}\label{sec:outline}
\subsubsection{Optimality conditions for Parisi measures} 

We begin our analysis by deriving necessary and sufficient conditions for a measure $\mu$ to satisfy
\[
\min_{\Pr\left[0,1\right]}\,\cP=\cP\left(\mu\right).
\]
To this end, we show that, in an appropriate sense, the variational derivative of $\cP$ is given by 
\[
\delta \cP_\mu= G_{\mu},
\]
where $G_\mu$ is given by (\ref{eq:Gmu}). Under this interpretation of $G$, one readily gets (\ref{eq:EL-rel}) from the first order optimality condition:
\[
\innerp{\delta P_\mu,\sigma}\geq0
\]
for all $\sigma=\tilde{\mu}-\mu$, $\tilde{\mu}\in\Pr[0,1]$. The remainder of Proposition \ref{prop:opt-conds} follows from basic principles.
This is presented in Sect.\ \ref{sec:First-Var}.

\subsubsection{A shift of viewpoint}\label{subset:change-coords}

For the remainder of this section, we focus on the question of the RS phase for mixed p-spin glass models, and in particular on the resolution of Conjecture \ref{conj:ATline}. 

It is natural to expect ``high temperature''-like (RS) behavior when $\beta$ is small and $h$ is large. (This is proven in Section \ref{sec:elementaryRS}.)
 The main difficulty is  to 
understand the region near the phase boundary where $\beta$ is large and $h$ is moderate. At the heart of our analysis is the idea 
that one should study the problem along the level sets $\alpha=$ const. This provides us with a useful scaling relation between $
\beta$ and $h$ which allows us to probe the region up to the phase boundary even in the regime where $\beta,h$ are large.

\subsubsection{Dispersive estimates of Gaussians}

With the above discussion in mind, we see that we will need good control of physical quantities, e.g.\ $q_*$, when $
\alpha=$ const.\ as $\beta$ becomes large. Such quantities will generically be given by expectations of functions of the Auffinger-Chen process
(\ref{eq:AC-SDE}). We thus need estimates of such expectations as the variance and mean of the process, effectively given by $\beta^2$ and 
$h$ respectively, diverge. Our main tool will be a technique to develop such estimates, 
which we call \emph{dispersive estimates} of Gaussians.

Dispersive estimates of Gaussians study the asymptotics of Gaussian integrals of the form 
\[
\E f(X_t)
\]
with $X_t\sim\cN(\mu(t),\Sigma(t))\in \R^d$, where some of the eigenvalues of $\Sigma(t)$ diverge in the limit $t\to\infty$.   The main idea is best encapsulated
by the simple observation that as $\sigma\rightarrow\infty$,
\[
\sigma\frac{e^{-\frac{x^2}{2\sigma^2}}dx}{\sqrt{2\pi\sigma^2}}\rightarrow\frac{dx}{\sqrt{2\pi}}
\]
in the sense of distributions. The goal of dispersive estimates is to quantify the rate of convergence of
\[
\innerp{f,\sigma\frac{e^{-x^2/2\sigma}dx}{\sqrt{2\pi\sigma^2}}}\rightarrow \innerp{f,\frac{dx}{\sqrt{2\pi}}},
\]
under minimal assumptions on $f$. Notice that if there is a mean $h$ which is also diverging, then the curves
$(\sigma,h(\sigma))$ of the form 
\[
\sigma e^{-\frac{1}{2}\left(\frac{h}{\sigma}\right)^{2}}=\text{constant}
\]
are identified. For the problem at hand, these curves are asymptotically of the form $\alpha=const.$ 

The main technical difficulty arises in higher dimensions. Here we will have a one parameter family of such 
processes $(X^\theta_t)$ where the non-diverging eigenvalues of $\Sigma^\theta$ will vanish as $\theta\rightarrow0$. The goal 
will be to obtain estimates that are uniform in $\theta$. This is explained in Sect.\ \ref{sec:disp-gauss}.

\subsubsection{The long time argument}
We now outline the proof of Theorem \ref{thm:long-time}. 
We begin by manipulating
$G_{\delta_{q_*}}$ to show that if 
\[
\xi''\left(y\right)\E_{h}\left[u_{xx}^{2}\left(y,X_{y}\right)\right]-1\leq 0 \qquad \forall y\geq q_*,
\]
then the model is RS where $u$ is the Parisi PDE solution corresponding to $\delta_{q_*}$ and $X_t$ is the corresponding SDE solution. Using It\^o's lemma for $u_{xx}^2$, we then show, after further manipulation,
that it suffices to show that
\[
\E_h\left(4 \sech^4(X_t)-6\sech^6(X_t)\right)\leq -c\cdot O(1/\beta^2)
\]
uniformly for $t\in[q_*,1]$. That the expectation is negative to order $\beta^{-2}$  can be seen to be true at the point $t=q_*$ 
using the 1-d dispersive estimates in
Sect.\ \ref{sec:disp-gauss}. In order to obtain this uniform estimate, however, one needs the full power of 2-d dispersive
estimates. After
re-parameterizing the interval $t\in [q_*,1]$ in such a way that it remains constant along the level sets, 
 we  obtain this result using the estimates from Sect.\ \ref{sec:2d-disp}. This allows us to compare the problem along the level sets with a ``limiting'' problem at $\beta = \infty$. These results
are presented in Sect.\ \ref{sec:long-time}.

\subsubsection{The 2/3-AT line argument}\label{sec:23-ATline} 
We now outline the proof of Proposition \ref{prop:23-AT}. Instead of studying the sets $\alpha= const.$, we consider the region $\alpha\leq 2/3-o_{\beta}(1)$. 
Let $q_*$ be as in (\ref{eq:q-alpha-def}). 
By manipulating the expression for $G_{\delta_{q_*}}$, we show
in Lemma \ref{lem:23-reduction} that the region of phase space for which 
\[
 \xi''(y)(1-q_*) \leq 1 \quad \forall\,y\geq q_*
\]
is in RS. Using the 1-d dispersive estimates from Sect.\ \ref{sec:disp-gauss}, we identify the left hand side at $y=q_*$ as $\frac{3}{2} \alpha$ to
leading order in $\beta$. Thus in a region $\alpha\leq 2/3 - o_\beta(1)$, we can conclude Proposition \ref{prop:23-AT}.
These results are presented in Sect.\ \ref{sec:23-arg}.

\subsubsection{Outline of paper}
The remainder of the paper is organized as follows. In Sect.\ \ref{sec:First-Var} we present the derivation of the optimality
conditions. In Sect.\ \ref{sec:RSprelims}, we present preliminary results on the RS phase. In Sect.\ \ref{sec:disp-gauss},
we present the dispersive estimates. In Sect.\ \ref{sec:23-arg}, we present the ``2/3 argument''. In Sect.\ \ref{sec:long-time} we present
the ``long time'' argument. Sect.\ \ref{sec:appendix} is an appendix which contains various technical remarks.

\section{Background and relation to previous results}\label{sec:background}
The importance of (\ref{eq:parisi-formula}) comes from the study of mixed $p$-spin glass models, which are defined as follows.
Consider the hypercube $\Sigma_N=\{-1,1\}^N$ with Hamiltonian
\[
H_N(\sigma)=H_N'(\sigma)+h\sum_{i\leq N}\sigma_i
\]
where $H'_N$ is the centered Gaussian process on $\Sigma_N$ with covariance
\[
\E H_N'(\sigma^1)H_N'(\sigma^2)=N\xi((\sigma^1,\sigma^2)/N))
\]
where $\xi=\beta^2\xi_0$ and $\xi_0$ is the model. Define the corresponding partition function
\[
Z_N=\sum_{\sigma}e^{-H_N(\sigma)}
\]
and Gibbs measure
\[
G_N(\sigma)=\frac{e^{-H_N(\sigma)}}{Z_N}.
\]

It was predicted by Parisi and proven by  Talagrand \cite{TalPF} and Panchenko \cite{PanchPF14}
 that the thermodynamic limit of the intensive free energy of a mixed $p$-spin glass is
given by the  variational formula 
\begin{equation}\label{eq:parisi-formula-2}
\lim_{N\rightarrow\infty}\frac{1}{N}\log Z_N=\min_{\mu\in\Pr([0,1])}\cP(\mu;\xi_0,\beta,h) +\log 2 \qquad a.s.
\end{equation}
For a concise introduction to the proof of this formula, see \cite{Panch12,PanchSKbook}.

The minimizer of this formula is thought of as the order parameter of the system and is expected to be (related
to) the limiting mean measure of the overlap $R_{12}$ which is defined by
\[
R_{12}=\frac{(\sigma^1,\sigma^2)}{N}
\]
where $\sigma^1$ and $\sigma^2$ are drawn independently from $ G_N$ are called \emph{replica}. This conjectured relationship can be 
proven for a large class of models \cite{PanchSKbook}. For the particular case of the RS regime,
this is known in wider generality. For example, if at least two of the $\beta_p's$ are non-zero with even $p$,
this follows from an integration by parts \cite[Theorem 3.7]{PanchSKbook}, Jensen's inequality, and a version of Talagrand's
positivity principle \cite[Chap. 14.12]{TalBK11vol2}. In the case of SK, where our theorem applies, this is a consequence of the same integration by parts argument combined with \cite[Theorem 13.7.8]{TalBK11vol2}.

A remarkable property of spin glasses which is exhibited in these systems is the breakdown 
of strong law of large numbers type behavior at low temperature of the overlap. (For an example of this in a more classical probabilistic
setting see \cite{BABM05,BovKurk04,Bov12}.) At sufficiently high temperature and external field, one expects the 
 overlap distribution to be a Dirac mass at a point depending on $\beta$ and $h$ by the law of large numbers. 
At sufficiently low temperature and external field, however,
one expects the model to have unusual behavior: the limiting overlap distribution should be non-degenerate.
This breakdown of the law of large numbers is called ``the breaking of the replica symmetry''  in the physics 
literature \cite{MPV87}.

The RS to RSB transition is the most studied aspect of the phase diagram in the 
literature, and constitutes the bulk of this paper. In \cite{AT78}, de Almeida and Thouless performed a stability analysis of the Replica Symmetric solution of the Sherrington-Kirkpatrick (SK) model, and concluded that it
is valid in a region 
defined by the self-consistency conditions from Proposition \ref{prop:opt-conds}. In particular, they concluded that the phase boundary is given by the curve
$\alpha = 1$, which is now called the AT line. 

The rigorous study of this phase transition was  initiated by Aizenman, Leibowitz, and Ruelle in \cite{ALR87}
where they showed that at sufficiently high temperature ($\beta\leq 1$) and zero external field ($h=0$) the SK model is in the RS phase. 
In \cite{Ton02}, Toninelli showed that below the AT line, the model is RSB. Guerra \cite{Guer01}
and Guerra-Toninelli \cite{GuerTon02} showed that for a generalization of the SK model, the RS phase is 
contained in the region above the AT line by deriving a Hamilton-Jacobi equation for the free energy in certain coefficients of the Hamiltonian.  Talagrand showed \cite{Tal02} that the model is in the RS phase in a region whose boundary is conjectured
to coincide with the AT line, using a technique which is related to his technique of 2-d Guerra bounds. 
An alternative characterization of the RS phase by Talagrand \cite{TalBK11vol2} is the region where the minimum value of the Parisi functional among one- and two-atomic measures is the same. In light of \cite{AuffChen14}, this is a clear consequence of the strict convexity of the Parisi functional \cite{Chen15}.
An analogous question 
was studied for related models in \cite{Panch05}. All of these aforementioned results, with the exception of \cite{Chen15,Ton02}, were from the perspective of 
calculating the left hand side of (\ref{eq:parisi-formula-2}) and did not use the variational formulation since the equality (as well as the strict convexity) 
had yet to be rigorously proven.

The bulk of this paper is devoted to the study of the generalization of the AT line conjecture described above
 for general mixed p-spin glass models 
which reduces to the AT line problem for the cases mentioned above.  Theorem \ref{thm:long-time} 
shows that the AT line is exactly the phase boundary
at sufficiently low temperature for all models including, for example, the SK model. 
 Proposition \ref{prop:23-AT} shows that a large portion of the region above the AT line is Replica Symmetric at moderate
and high temperatures. In the case of the SK model, Theorem \ref{thm:SK-32} coupled with the above verifies RS in the region above the AT line less a bounded set which does not contain the critical point at zero external field.

\section{First variation formula and optimality conditions}\label{sec:First-Var}

We derive necessary and sufficient conditions for a measure $\mu$ to satisfy
\[
\min_{\Pr\left[0,1\right]}\,\cP=\cP\left(\mu\right).
\]
To do so, we make precise the idea that if $\mu_t$ is a path in $\Pr[0,1]$
starting at the minimizing measure $\mu_0=\mu$ then 
\[
\frac{d}{dt}^+ \cP(\mu_0) \geq 0.
\] 
After calculating the first-variation $\delta\cP$, we derive optimality conditions which characterize Parisi measures,
as well as self-consistency conditions for a model to be kRSB.
As an application, we prove that all models are RSB below the (generalized) AT line.

\subsection{First variation of the Parisi functional}

In the following, we work with the weak topology on $\cM\left[0,1\right]$, the
Radon measures. We metrize $\Pr\left[0,1\right]$ in the relative
topology using the metric
\[
d\left(\mu,\nu\right)=\int_{0}^{1}\abs{\mu\left[0,s\right]-\nu\left[0,s\right]}\, ds.
\]
Let us denote the duality pairing of $C\left[0,1\right]$ with $\cM\left[0,1\right]$
by 
\[
\left\langle \phi,\mu\right\rangle =\int_{\left[0,1\right]}\phi\, d\mu.
\]

\begin{definition}
We call a path of measures $\mu_{t}:\left[0,1\right]\to\cM\left[0,1\right]$
\emph{weakly differentiable }if the weak limit 
\[
\lim_{\epsilon\to0}\,\frac{1}{\epsilon}\left(\mu_{t+\epsilon}-\mu_{t}\right)=\dot{\mu}_{t}\in\cM
\]
exists for all $t\in\left(0,1\right)$, in which case we call $\dot{\mu}_{t}$
the \emph{weak derivative }of $\mu_{t}$. We call $\mu_{t}$ \emph{right
weakly differentiable at $t=0$ }if the weak right limit
exists.
\end{definition}
We now compute the first variation of $\cP$.  The motivation for the proof is as follows.
Since the Parisi functional is the sum of the Parisi PDE solution $u$ corresponding to $\mu$
and a linear term,
\[
\cP\left(\mu\right)=u\left(0,h\right)-L\left(\mu\right),
\]
the difficulty lies only with understanding $u$.
Consider the formal variation of the solution
$\delta u$ with respect to a variation in the measure $\delta\mu$.
Differentiating the Parisi PDE in $\mu$ we find
\[
\left(\partial_{t}+\cL_{t,\mu}\right)\delta u+\frac{\xi''}{2}\delta\mu u_{x}^{2}=0,
\]
where $\cL_{t,\mu}$ is the infinitesimal generator for the Auffinger-Chen SDE (\ref{eq:AC-SDE}) with measure $\mu$. 
One then recovers $\delta u$ using It\^o's lemma, and a rearrangement and integration by parts then
suggests the formula for $G$ from (\ref{eq:Gmu}).

We will need the following notation: if a function $f:\left[0,1\right]\to\R$
is right differentiable at $x\in[0,1)$ we denote the right derivative
as
\[
\frac{d}{dx}^{+}f\left(x\right)=\lim_{y\to x^{+}}\,\frac{f\left(y\right)-f\left(x\right)}{y-x}.
\]
We denote the left derivative similarly by $\frac{d}{dx}^- f$.

\begin{lemma}\label{lem:right-derivP}
Let $\mu_{t}:\left[0,1\right]\to\Pr\left[0,1\right]$
be weakly differentiable. Then the function $t\to\cP\left(\mu_{t}\right)$
is differentiable, and 
\[
\frac{d}{dt}\cP\left(\mu_{t}\right)=\left\langle G_{\mu_{t}},\dot{\mu}_{t}\right\rangle 
\]
for all $t\in\left(0,1\right)$. If $\mu_{t}$ is right weakly differentiable
at $t=0$ then $t\to\cP\left(\mu_{t}\right)$ is right differentiable
at $t=0$ and 
\[
\frac{d}{dt}^{+}\cP\left(\mu_{0}\right)=\left\langle G_{\mu_{0}},\dot{\mu}_{0}\right\rangle .
\]
\end{lemma}
\begin{proof}
We prove the result by establishing an inequality of the form
\[
\abs{\cP\left(\tilde{\mu}\right)-\cP\left(\mu\right)-\left\langle G_{\mu},\tilde{\mu}-\mu\right\rangle }\leq Cd^{2}\left(\tilde{\mu},\mu\right)
\]
for all $\tilde{\mu},\mu\in\Pr\left[0,1\right]$. In the following, $C(\xi)$ will denote a constant depending only on $\xi$ which may change between lines.

With this in mind, let $\tilde{\mu},\mu\in\Pr\left[0,1\right]$ and
$\tilde{u}$, $u$ be the corresponding Parisi PDE solutions. Let
$X_{s}$ solve the Auffinger-Chen SDE corresponding to $\mu$ and let $\cL_{t,\mu}$ be the infinitesimal generator. 
Then if $\delta=\tilde{u}-u$,
\[
\left(\partial_{t}+\cL_{t,\mu}\right)\delta=-\frac{\xi''}{2}\left(\left(\tilde{\mu}[0,t]-\mu[0,t]\right)\tilde{u}_{x}^{2}+\mu[0,t]\delta_{x}^{2}\right)
\]
weakly with final time data $\delta\left(1,x\right)=0$. By the regularity
given in  Sect.\ \ref{sec:appendixPDE} we have the representation 
\[
\delta\left(0,h\right)=\E_{X_{0}=h}\int_{0}^{1}\frac{\xi''\left(s\right)}{2}\left[\left(\tilde{\mu}\left[0,s\right]-\mu\left[0,s\right]\right)\tilde{u}_{x}^{2}\left(s,X_{s}\right)+\mu\left[0,s\right]\delta_{x}^{2}\left(s,X_{s}\right)\right]\, ds.
\]
Therefore by Fubini's theorem,
\begin{align*}
&\delta\left(0,h\right) - \innerp{\int_t^1 \frac{\xi''(s)}{2}\E_{X_0=h}\left(u_x^2(s,X_s)\right)\,ds,\tilde{\mu}-\mu} \\
&\quad=\E_{X_{0}=h}\int_{0}^{1}\frac{\xi''\left(s\right)}{2}\left[\left(\tilde{\mu}\left[0,s\right]-\mu\left[0,s\right]\right)
\left(\tilde{u}_{x}^{2}-u_x^2\right)\left(s,X_{s}\right)+\mu\left[0,s\right]\delta_{x}^{2}\left(s,X_{s}\right)\right]\, ds.
\end{align*}
Note that the results in  Sect.\ \ref{sec:appendixPDE} give that
\[
\norm{\tilde{u}_{x}^{2}-u_{x}^{2}}_{\infty}\vee\norm{\delta_{x}}_{\infty}\leq C\left(\xi\right)d\left(\tilde{\mu},\mu\right),
\]
so that by the triangle inequality we have that
\[
\abs{\delta\left(0,h\right) - \innerp{\int_t^1 \frac{\xi''(s)}{2}\E_{X_0=h}\left(u_x^2(s,X_s)\right)\,ds ,\tilde{\mu}-\mu}}\leq C\left(\xi\right)d^{2}\left(\tilde{\mu},\mu\right).
\]
Applying Fubini's theorem to the linear term in the Parisi functional, we have that
\[
L\left(\mu\right)=  \innerp{\int_{t}^{1}\frac{\xi''\left(s\right)}{2}s\, ds , \mu}.
\]
Therefore, by the definition of $G_{\mu}$ in (\ref{eq:Gmu}) and the linearity of $L$, we have the inequality
\[
\abs{\cP\left(\tilde{\mu}\right)-\cP\left(\mu\right)-\left\langle G_{\mu},\tilde{\mu}-\mu\right\rangle }\leq C\left(\xi\right)d^{2}\left(\tilde{\mu},\mu\right).
\]

Now we prove the claims. Given $\mu_{t}:\left[0,1\right]\to\Pr\left[0,1\right]$
we have that
\[
\abs{\frac{1}{\epsilon}\left(\cP\left(\mu_{t+\epsilon}\right)-\cP\left(\mu_{t}\right)\right)-\left\langle G_{\mu_{t}},\frac{\mu_{t+\epsilon}-\mu_{t}}{\epsilon}\right\rangle }\leq C\left(\xi\right)\frac{1}{\epsilon}d^{2}\left(\mu_{t+\epsilon},\mu_{t}\right)
\]
by the above. Since the weak convergence of $\left(\mu_{t+\epsilon}-\mu_{t}\right)/\epsilon$
in $\cM$ implies the bound
$d\left(\mu_{t+\epsilon},\mu_{t}\right)/\epsilon\leq C$,
we immediately conclude that
\[
\lim_{\epsilon\to0}\,\frac{1}{\epsilon}\left(\cP\left(\mu_{t+\epsilon}\right)-\cP\left(\mu_{t}\right)\right)=\left\langle G_{\mu_{t}},\dot{\mu}_{t}\right\rangle .
\]
The proof of right-differentiability at $t=0$ is the same.
\end{proof}

\begin{definition}
We call the function $G_{\mu}$ defined in (\ref{eq:Gmu}) the
\emph{first-variation of $\cP$ at $\mu$}, and we write
\[
\delta\cP\left(\mu\right)=G_{\mu}.
\]
If $\mu_{t}:\left[0,1\right]\to\Pr\left[0,1\right]$ is weakly differentiable (right weakly differentiable at $t=0$)
we refer to 
\[
\delta_{\dot{\mu}_{t}}\cP\left(\mu_{t}\right)=\left\langle G_{\mu_{t}},\dot{\mu}_{t}\right\rangle 
\]
as the (one-sided) \emph{variation of $\cP$ at $\mu_{t}$ in the direction of
$\dot{\mu}_{t}$. }
\end{definition}
We finish this section with a particularly useful example. 
\begin{example}
\label{ex:mixing}Let $\mu,\tilde{\mu}\in\Pr\left[0,1\right]$ and define
the mixing variation
\[
\mu_{\theta}=\mu+\theta\left(\tilde{\mu}-\mu\right),\quad\theta\in\left[0,1\right].
\]
The path $\mu_{\theta}$ is linear and therefore weakly differentiable
(right weakly differentiable at $\theta=0$), and the weak derivative (right
weak derivative) is given by
\[
\dot{\mu}_{\theta}=\tilde{\mu}-\mu,\quad \theta\in[0,1).
\]

\end{example}

\subsection{Optimality conditions}
We establish necessary conditions on the first-variation $\delta \cP$ at a minimizing measure. 
As the Parisi functional is convex, these conditions are also sufficient. 
\begin{lemma}
\label{lem:onesidedEL} The measure $\mu\in\Pr\left[0,1\right]$ minimizes the
Parisi functional if and only if for every right weakly differentiable
path $\mu_{t}$ with $\mu_{0}=\mu$, $\dot{\mu}_{0}=\sigma$ we have
\[
\delta_{\sigma}\cP\left(\mu\right)\geq0.
\]
Furthermore, $\mu$ minimizes the Parisi functional if and only if for every mixing variation
 
 \[
\mu_{\theta}=\mu + \theta\left(\tilde{\mu}-\mu\right),\ \theta\in\left[0,1\right],\ \tilde{\mu}\in\Pr[0,1],  
 \]
 we have
 \[
  \delta_{\tilde{\mu}-\mu}\cP\left(\mu\right)\geq0.
 \]
\end{lemma}
\begin{proof}
Suppose that $\mu$ minimizes $\cP$, i.e.,
\[
\cP\left(\tilde{\mu}\right)-\cP\left(\mu\right)\geq 0
\]
for all $\tilde{\mu}\in\Pr\left[0,1\right]$. Let $\mu_{t}$ be right
weakly differentiable with $\mu_{0}=\mu$ and $\dot{\mu}_{0}=\sigma$,
then by Lemma \ref{lem:right-derivP} and the definition of $\delta\cP$
we find that
\[
\delta_{\sigma}\cP\left(\mu\right)=\frac{d}{dt}^+\cP\left(\mu_0\right)\geq 0.
\]

On the other hand, if $\mu$ does not minimize $\cP$ there exists a distinct $\tilde{\mu}\in\Pr\left[0,1\right]$
with
\[
\cP\left(\tilde{\mu}\right)-\cP\left(\mu\right)<0.
\]
Consider the mixing variation $\mu_{\theta}$ defined above, and note that $\dot{\mu}_0 = \tilde{\mu}-\mu$.
By Lemma \ref{lem:right-derivP}, the function $\theta\to\cP\left(\mu_{\theta}\right)$ is right differentiable at $\theta=0$. By the convexity of $\cP$ and the linearity of this variation, we see that 
$\theta\to\cP\left(\mu_{\theta}\right)$ is convex. It immediately follows that
\[
\frac{d}{d\theta}^{+}\cP\left(\mu_0\right)=\inf_{\theta\in(0,1]}\,\frac{\cP\left(\mu_{\theta}\right)-\cP\left(\mu\right)}{\theta},
\]
 and hence that 
\[
\delta_{\tilde{\mu}-\mu}\cP\left(\mu\right)=\frac{d}{d\theta}^{+}\cP\left(\mu_{0}\right)\leq\cP\left(\tilde{\mu}\right)-\cP\left(\mu\right)<0.
\]
\end{proof}
We refer to the following as the optimality conditions which $\mu$ must satisfy to be the Parisi
measure.
\begin{corollary}[optimality conditions] \label{cor:optimality}
The measure $\mu\in\Pr\left[0,1\right]$ minimizes the Parisi functional if and
only if 
\[
\mu\left(\left\{ t\,:\,\text{min}\, G_{\mu}=G_{\mu}\left(t\right)\right\} \right)=1.
\]
\end{corollary}
\begin{proof} 
By Lemma \ref{lem:onesidedEL} and the definition of $\delta P(\mu)$,
we see that $\mu\in\Pr[0,1]$ minimizes the Parisi functional if and only if for every $\tilde{\mu}\in\Pr[0,1]$ we have that
 \[
  0\leq \innerp{G_{\mu},\tilde{\mu}-\mu}=\innerp{G_{\mu},\tilde{\mu}} - \innerp{G_{\mu},\mu}.
 \]
 The claim follows immediately.
\end{proof}

As a result of Corollary \ref{cor:optimality}, we can prove that  
the class of mixing variations involving adding a single atom is enough to test for optimality.
\begin{corollary} \label{cor:add1atom}
 The measure $\mu\in\Pr\left[0,1\right]$ minimizes the Parisi functional if and only if for every mixing variation of the form
 $\mu_{\theta}=\mu + \theta\left(\delta_q-\mu\right)$, $\theta\in\left[0,1\right]$, $q\in[0,1]$, we have
 \[
  \delta_{\delta_q-\mu}\cP\left(\mu\right)\geq0.
 \]
\end{corollary}
\begin{proof}
 By the definition of $\delta P(\mu)$, the claim is that $\mu$ minimizes if and only if
 \[
  0\leq \innerp{G_\mu,\delta_q - \mu} = G_\mu(q) - \innerp{G_\mu,\mu}\quad \forall\,q\in[0,1],
 \]
and this is equivalent to the statement given in Corollary \ref{cor:optimality}.
\end{proof}


\subsection{Self-consistency conditions for minimizers}
We give two preliminary results on the support of the minimizing measure. Then we derive self-consistency conditions for Parisi measures.

\begin{lemma}\label{lem:1notsupp}
$1$ is not in the support of the minimizer.\end{lemma}
\begin{proof}
If $1$ is in the support, then $G_{\mu}\left(1\right)=\min\, G_{\mu}$ by Corollary \ref{cor:optimality}
so that $\frac{d}{dt}^- G_{\mu}\left(1\right)\leq0$. By the definition of $G_{\mu}$ in (\ref{eq:Gmu}), we find that
\[
\frac{\xi^{\prime\prime}(1)}{2}(\E_{h}\tanh^{2}(X_{1})-1)\geq0,
\]
which is absurd.
\end{proof}
\begin{lemma}\label{lem:0notsupp}
If $h\neq0$, then $0$ is not in the support of the minimizer. In fact, if $\mu$ is minimizing 
\[
u_{x}^{2}\left(0,h\right)\leq\inf\text{supp}\,\mu.
\]
\end{lemma}
\begin{proof}
Given the inequality, we observe that  $h\neq0$ implies $0\notin\text{supp}\,\mu$. Indeed, by even symmetry of 
$u(t,\cdot)$ we have 
that $u_x(0,0)=0$, and by the results of Sect.\ \ref{sec:appendixPDE} we have that $u_{xx}>0$.

Now we prove the inequality. Call $y=\inf\text{supp}\,\mu$. By Corollary \ref{cor:optimality}, we have that
$G_{\mu}\left(y\right)=\min\, G_{\mu}$, and therefore
\[
G_{\mu}\left(y\right)-G_{\mu}\left(y+\epsilon\right)\leq0
\]
for sufficiently small $\epsilon > 0 $. By the definition
of $G_{\mu}$, 
\[
\fint_{y}^{y+\epsilon}\frac{\xi''\left(s\right)}{2}\left(\E_{h}u_{x}^{2}\left(s,X_{s}\right)-s\right)\, ds\leq0,
\]
hence there exists $t\in\left(y,y+\epsilon\right)$ with
\[
\E_{h}u_{x}^{2}\left(t,X_{t}\right)\leq t.
\]
Using It\^{o} calculus (see Sect.\ \ref{sec:AC-SDE}), we have that
\[
\frac{d}{ds}\E_{h}u_{x}^{2}\left(s,X_{s}\right)=\xi''\left(s\right)\E_{h}u_{xx}^{2}\left(s,X_{s}\right)\geq0,
\]
hence
\[
u_{x}^{2}\left(0,h\right)\leq y+\epsilon
\]
for sufficiently small $\epsilon>0$. This proves the result.\end{proof}

Using these results, we can derive the following set of self-consistency conditions that Parisi measures must satisfy.
Note by the definition of $G_\mu$ in (\ref{eq:Gmu}) and It\^{o} calculus (see Sect.\ \ref{sec:AC-SDE})
we know that $G_{\mu}\in C^2$.

\begin{corollary}[Consistency conditions]\label{cor:consistency}
If $\mu$ minimizes the Parisi functional, 
\begin{align*}
 G_{\mu}'\left(y\right)=0 \qquad
 G_{\mu}''\left(y\right)\geq0
\end{align*}
and
\begin{align*}
\E_{h}\left[u_{x}^{2}\left(y,X_{y}\right)\right]  =y \qquad
\xi''\left(y\right)\E_{h}\left[u_{xx}^{2}\left(y,X_{y}\right)\right]  \leq1
\end{align*}
for all $y\in\text{supp}\,\mu$. (At $y=0$ the derivatives are understood to be right-derivatives.) \end{corollary}

\begin{remark} This result can be used to generate self-consistency conditions for a model to be kRSB.
As the solution to the Parisi PDE can be described explicitly in the case of k-atomic measures via the Cole-Hopf transformation, 
in principle these conditions can be checked. We discuss the simplest case $k=1$ in greater detail in Sect.\ \ref{sec:GenATLine}.
\end{remark}

\begin{proof}
Let $\mu$ be minimizing and recall that by Corollary \ref{cor:optimality}, we know $G_\mu$ is minimized on $\text{supp}\,{\mu}$. Using It\^{o} calculus
(see Sect.\ \ref{sec:AC-SDE}) we have that
\begin{align*}
G_{\mu}'\left(y\right) & =-\frac{\xi''\left(y\right)}{2}\left(\E_{h}\left[u_{x}^{2}\left(y,X_{y}\right)\right]-y\right) \\
G_{\mu}''\left(y\right)& =-\frac{\xi'''\left(y\right)}{2}\left(\E_{h}\left[u_{x}^{2}\left(y,X_{y}\right)\right]-y\right)-\frac{\xi''\left(y\right)}{2}\left(\xi''\left(y\right)\E_{h}\left[u_{xx}^{2}\left(y,X_{y}\right)\right]-1\right)
\end{align*}
for all $y\in[0,1]$, where in the cases $y=0$ and $y=1$ we understand the derivatives as right and left derivatives respectively. Therefore
the claims follow immediately at every $y\in\text{supp}\,\mu\cap(0,1)$ and we
only need to check the cases $y=1\in\text{supp}\,\mu$ and $y=0\in\text{supp}\,\mu$.  

By Lemma \ref{lem:1notsupp} the case $y=1$ never occurs. Let $y=0\in \text{supp}\,\mu$, then by Lemma \ref{lem:0notsupp} we have that $h=0$. Therefore
\[
\frac{d}{dy}^{+}G_{\mu}\left(0\right)=-\frac{\xi''\left(0\right)}{2}u_{x}^{2}\left(0,0\right)=0
\]
as desired, while 
\[
\frac{d}{dy}^{+}\left(\frac{d}{dy}^{+}G_{\mu}\right)\left(0\right) =-\frac{\xi''\left(0\right)}{2}\left(\xi''\left(0\right)u_{xx}^{2}\left(0,0\right)-1\right).
\]
Since $G_\mu(0)=\min\,G_\mu$ and $\frac{d}{dy}^{+}G_{\mu}(0)=0$ we have that $\frac{d}{dy}^{+}\left(\frac{d}{dy}^{+}G_{\mu}\right)(0)\geq 0$. Therefore,
\[
\xi''\left(0\right)u_{xx}^{2}\left(0,0\right)\leq1.
\]
\end{proof}
\subsection{Proofs of Proposition \ref{prop:opt-conds} and Theorem \ref{thm:RSBalphabigger1}} \label{sec:GenATLine}

We now prove Proposition \ref{prop:opt-conds} and Theorem \ref{thm:RSBalphabigger1}. We begin with the first, which we restate 
for the convenience of the reader.
\begin{proposition*}\textbf{\ref{prop:opt-conds}}
$\mu$ is a Parisi measure if and only if
\begin{equation*}
\mu(\{t:G_{\mu}(t)=\min G_{\mu}\})=1.
\end{equation*}
Furthermore, if $\mu$ is a Parisi measure, it must satisfy the self-consistency conditions,
\begin{equation*}
\begin{cases}
\E_h u_{x}^2(q,X_q)&=q\\
\xi''(q)\E_h u_{xx}^2(q,X_q)&\leq 1
\end{cases}
\end{equation*}
for all $q\in \text{supp}\,\mu$.
\end{proposition*}
\begin{proof}
This follows immediately from combining Corollary \ref{cor:optimality} and Corollary \ref{cor:consistency}.
\end{proof}

We turn to Theorem \ref{thm:RSBalphabigger1}. 
We will require the following facts.
\begin{fact}\label{fact:explicitsolns}
For all $t\geq \sup \text{supp}\,\mu$, the solution $u$ of the Parisi PDE satisfies
\begin{align*}
u(t,x)=\log\cosh(x)+\frac{1}{2}\left(\xi'(1)-\xi^{\prime}(t)\right)& \qquad 
u_x(t,x) = \tanh(x) \\
u_{xx}\left(t,x\right)  =\sech^{2}(x)&\qquad
u_{xxx}\left(t,x\right)  =-2\tanh\, x\cdot\text{sech}^{2}(x)
\end{align*}
For all $t\leq \inf \text{supp}\,\mu$, the solution $X_t$ of the Auffinger-Chen SDE with initial data $X_0=h$ satisfies
\[
 X_t = h+\int_0^t \sqrt{\xi''(s)} dW_s.
\]
\end{fact}

\begin{theorem*}\textbf{\ref{thm:RSBalphabigger1}} For any model $\xi_0$, $RS\subset AT$. 
\end{theorem*}

\begin{proof}
If the minimizer is $\mu=\delta_q$, then by the consistency conditions Corollary \ref{cor:consistency} and Fact \ref{fact:explicitsolns}
we find that
\begin{align*}
 \E\tanh^2( \sqrt{\xi'(q)}Z+ h) &=q \\
 \xi''\left(q\right)\E\text{sech}^4( \sqrt{\xi'(q)}Z+ h) &\leq 1.
\end{align*}

Therefore by the definitions of $Q_*$ and $\alpha$ in (\ref{eq:q-alpha-def}), we conclude that
$q\in Q_*$ and $\alpha \leq 1$. This completes the proof.
\end{proof}

\section{The RS phase} \label{sec:RSprelims}
In this section we present a preliminary analysis of Conjecture \ref{conj:ATline}. We begin by presenting some
reductions of the question. We then analyze the problem for moderate temperatures. 

Recall the definition of $q_*$ and $\alpha$ from (\ref{eq:q-alpha-def}).
By the optimality conditions in Corollary \ref{cor:optimality}, 
\[
G_{\delta_{q_*}}(q_*)=\min_y\, G_{\delta_{q_*}}(y)
\]
if and only if $\delta_{q_*}$ is the minimizer. 
To prove the required equality, we first note that $G_{\delta_{q_*}}$ is non-increasing on $[0,q_*]$. The problem reduces
to showing that $G_{\delta_{q_*}}$ is non-decreasing on $[q_*,1]$, which is implied by certain
conditions related to derivatives of $G_{\delta_{q_*}}$. 

\paragraph{Notation} Before we begin we introduce the following notation that will be used throughout Sects. 4, 6, and 7. Since
these sections refer only to the setting of RS, \textbf{we will
always take $\mu=\delta_{q_*}$ in these sections. 
To this end, we suppress the dependence of $u$, its derivatives, $G_{\delta_{q_*}}$, and  $X_t$ on $\delta_{q_*}$.} Furthermore, we will make frequent use of the 
following function:
\begin{equation}\label{eq:g-def}
g(y) = \E_{h}\left[u_{x}^{2}\left(y,X_{y}\right)\right]-y.
\end{equation}
We note here that as a consequence of It\^{o}'s lemma (Sect.\ \ref{sec:AC-SDE}),
\begin{equation}\label{eq:g'}
g'(y) = \xi''(y) \E_{h}\left[u_{xx}^{2}\left(y,X_{y}\right)\right]-1.
\end{equation}
Again, though the function $g$ can be defined for general $\mu$, for
the remainder of this paper it will \textbf{always}
be understood with $\mu=\delta_{q_*}$. For the reader more 
familiar with the notation of \cite{TalBK11,TalBK11vol2}, please see Sect. \ref{sec:appendix}, specifically \eqref{eq:g-def-alt}.

We begin with the following preliminary lemma.
\begin{lemma} Suppose $(\beta,h)\in AT$.
\begin{enumerate}
\item We have that
\[
G_{\delta_{q_{*}}}\left(y\right)\geq G_{\delta_{q_*}}\left(q_{*}\right),\quad y\leq q_{*}.
\]
\item If in addition $\alpha<1$, then $q_{*}$ is a local minimizer for $G_{\delta_{q_{*}}}$. 
\end{enumerate}
\end{lemma}
\begin{proof}
To prove 1, we will show that $G'\left(y\right)\leq0$ for $y\leq q_{*}$.
Note that by the definition of $g$ in \eqref{eq:g-def}, it suffices to show that
\[
g(y)\geq 0 \qquad \forall y\leq q_{*}.
\]
Since $g(q_{*})=0$, it is enough to check that $g'\leq0$
for $y \leq q_{*}$. 
By \eqref{eq:g'} and It\^o's lemma (Sect. \ref{sec:AC-SDE}), we have that
\[
g'\left(y\right) = \xi''(y) \E_h [u_{xx}^2(q_*,X_{q_*})] -\xi''(y)\int_y^{q_*} \xi''\left(t\right)\E_{h}\left[u_{xxx}^{2}\left(t,X_{t}\right)-2\mu\left[0,t\right]u_{xx}^{3}\left(t,X_{t}\right)\right]\, dt -1.
\]
Then, since $\alpha \leq 1$, $\xi''$ is increasing, $y\leq q_*$, and $\mu=\delta_{q_*}$, 
\begin{align*}
g'\left(y\right) &\leq \xi''(q_*) \E_h [u_{xx}^2(q_*,X_{q_*})] -1 -\xi''(y)\int_y^{q_*} \xi''\left(t\right)\E_{h}\left[u_{xxx}^{2}\left(t,X_{t}\right)\right]\, dt\\
&\leq  \xi''(q_*) \E_h [u_{xx}^2(q_*,X_{q_*})] -1 \leq 0.
\end{align*}

To prove 2, note that by the definition of G in (\ref{eq:Gmu}), and the definitions of $q_*$ and $\alpha$ in (\ref{eq:q-alpha-def}) we get that
\[
G_{\delta_{q_*}}'\left(q_{*}\right)  =0 \text{ and }
G_{\delta_{q_*}}''\left(q_{*}\right)  > 0.
\]
\end{proof}
By a similar argument we get the following potential strategies for studying the AT line conjecture.
\begin{lemma}\label{lem:strats}
If one of the following holds:
\begin{enumerate}
\item $g(y)\leq 0$ for all $y\geq q_*$
\item $g'(y)\leq 0$ for $y\geq q_*$,
\end{enumerate}
then
\[
G(y)\geq G(q_*), \qquad y\geq q_*.
\]
In particular, if $(\beta,h)\in AT$ and one of the above holds, then $(\beta,h)\in RS$.
\end{lemma}

\subsection{Preliminary results at moderate temperatures}
In this section we present some preliminary observations. 
\begin{lemma}\label{lem:beta-alr}
If $\xi''(1)\leq 1$, then $(\beta,h)\in RS$.
\end{lemma}
\begin{proof}
Recall the definition of $g$ in (\ref{eq:g-def}) and the formula for $g'$ given in \eqref{eq:g'}. Using
the fact that $\abs{u_{xx}}\leq 1$ (Sect.\ \ref{sec:appendixPDE}) and that $\xi''$ is non-decreasing, we conclude that
\begin{align*}
 g'(y) = \xi''(y)\E_h[u_{xx}^2(y,X_y)] -1 
       \leq \xi''(1) -1.
\end{align*}
The result follows from Lemma \ref{lem:strats}. 
\end{proof}
\begin{lemma}\label{lem:sk-3/2}
For the SK model, 
\[
\{(\beta,h):\alpha \leq 1,\ \beta \leq 3/2\} \subset RS.
\]
\end{lemma}
\begin{remark}
This shows that the AT line is the RS to RSB phase boundary for the SK model in the $(\beta,h)$-plane even for
$\beta\leq 3/2$. Note that this upper bound is larger than the critical temperature at $h=0$. This suggests that the 
$(\beta,h)=(1,0)$ critical point is not identified in these analyses for models with $\xi''(0)>0$. Indeed one can extend this argument
to such models in a case-by-case fashion. For example, it will hold if $\xi'''(1)$ is sufficiently small with respect to $\xi''(1)$.
We believe the criticality as $\xi''(0)\rightarrow0$ is highly nontrivial.
As the reader will see, the argument breaks down when $\xi''(0)=0$. 
\end{remark}
\begin{proof}
Recall that $\xi''(t) = \beta^2$ in the SK model. Let $g$ be as in (\ref{eq:g-def}), and recall by Lemma \ref{lem:strats} that it is enough to prove that $g'\leq0$ for $y\geq q_*$. 
By \eqref{eq:g'}, we have that $g' = \beta^2 f - 1$ where 
\[
f(y) = \E_h[u_{xx}^2(y,X_y)].
\]
Using It\^o's lemma (Sect.\ \ref{sec:AC-SDE}), the fact that $\mu = \delta_{q_*}$, and that $u$ satisfies the equations from Fact \ref{fact:explicitsolns}, 
\[
 f'(y) = \beta^2\E_{h}\left[4\text{sech}^{4}(X_{y})-6\text{sech}^{6}(X_{y})\right],\quad \forall\,y > q_*.
\]
It immediately follows from an application of Jensen's inequality that $f$ satisfies the differential inequality
\[
f'\leq 2 \beta^2 (2f-3f^{3/2}).
\]
Note the following comparison principle: if for each $x$, $y \to \phi(y,x)$ solves the ordinary differential equation 
\begin{equation}\label{eq:odecomparison}
\phi' = 2\beta^2 (2\phi - 3\phi^{3/2}) \quad \forall\,y> q_*
\end{equation}
with initial condition $\phi(q_*,x) = x$, then 
\[
f(y)\leq \phi(y,f(q_*)) \quad \forall\,y\geq q_*.
\] 

We now complete the proof by a case analysis. Suppose that $f(q_*)\leq 4/9$. Since the constant $4/9$ is a stationary solution of \eqref{eq:odecomparison},
the comparison principle gives that $f(y)\leq 4/9$ for $y\geq q_*$. Hence,
\[
g'(y) = \beta^2 f(y) - 1\leq \beta^2\frac{4}{9}-1\leq 0.
\]
Now suppose that $f(q_*)> 4/9$. Note that the solution $\phi(y,f(q_*))$ to \eqref{eq:odecomparison}  is non-increasing in $y$. Thus, the comparison principle implies that $f(y) \leq f(q_*)$ for all $y\geq q_*$. It follows from the definitions of $f$ and $\alpha$ that
\[
g'(y) \leq g'(q_*) = \alpha -1\leq 0.
\]
In either case, we have that $g'(y)\leq 0$ for all $y \geq q_*$.
\end{proof}

\section{Dispersive estimates of Gaussians}\label{sec:disp-gauss}
In the first subsection, we develop the dispersive estimates that we need when $d=1$. In the second subsection, we develop the dispersive estimates that we need
in $d=2$. In the following, we always use the Fourier transform with normalization
\[
\hat{f}(k)=\frac{1}{(2\pi)^{d/2}}\int f(x)e^{-i\innerp{k,x}}dx.
\]

\subsection{Dispersive estimates in 1-d}\label{sec:1d-disp}
In this subsection, we take $d=1$. We begin with the following soft lemma. 
\begin{lemma}\label{lem:softdisp}
Let $Z$ be a standard Gaussian and let $f\in L_1(dx)\cap L_2(dx)$. For any sequence
$(\sigma,h(\sigma))$ where $\sigma\rightarrow\infty$, 
\[
\limsup_{\sigma\rightarrow\infty}\sigma\E f(h+\sigma Z)=\frac{e^{-\frac{1}{2}\liminf\left(\frac{h}{\sigma}\right)^2}}{\sqrt{2\pi}}\int f(x)dx
\]
\end{lemma}
\begin{proof}
Note that
\[
\E f\left(h+\sigma Z\right)=\frac{1}{\sigma}\innerp{ e^{i\frac{h}{\sigma}x}\hat{f}\left(\frac{x}{\sigma}\right),\frac{e^{-x^{2}/2}}{\sqrt{2\pi}}}_{L^{2}\left(dx\right)}.
\]
Since $\hat{f}\in C_0(\R)$, we can apply the dominated
convergence theorem to conclude  that
\[
\limsup_{\sigma\to\infty}\,\left\langle e^{i\frac{h}{\sigma}x}\hat{f}\left(\frac{x}{\sigma}\right),\frac{e^{-x^{2}/2}}{\sqrt{2\pi}}\right\rangle _{L_2\left(dx\right)}=\limsup_{\sigma\to\infty}\,\hat{f}\left(0\right)\left\langle e^{i\frac{h}{\sigma}x},\frac{e^{-x^{2}/2}}{\sqrt{2\pi}}\right\rangle _{L_2\left(dx\right)}=\hat{f}\left(0\right)e^{-\frac{1}{2}\liminf\left(\frac{h}{\sigma}\right)^{2}}.
\]
\end{proof}
We now quantify the rate of convergence of 
$\sigma\E f\left(h+\sigma Z\right)$
as $\sigma\to\infty$. In the following, let 
$\cE=\{f: f\in L_1((x^2\vee 1)dx),\, f(x)=f(-x)\}.$
\begin{lemma}\label{lem:1d-disp}
Suppose that $f\in \cE$.
We have that
\[
\abs{\sigma\E f\left(h+\sigma Z\right)-\frac{e^{-\frac{1}{2}\left(\frac{h}{\sigma}\right)^{2}}}{\sqrt{2\pi}}\int f\left(x\right)\, dx}
\leq\frac{1}{2}\frac{1}{\sqrt{2\pi}}\frac{1}{\sigma^{2}}\norm{f}_{L_1(x^2dx)}.
\]
\end{lemma}
\begin{proof}
Note that by density, it suffices to check this for $f\in L_2(dx)\cap L_1((x^2\vee 1)dx)$. By the triangle inequality,
\[
\abs{\left\langle e^{i\frac{h}{\sigma}x}\hat{f}\left(\frac{x}{\sigma}\right),\frac{e^{-x^{2}/2}}{\sqrt{2\pi}}\right\rangle _{L_2\left(dx\right)}-\hat{f}\left(0\right)e^{-\frac{1}{2}\left(\frac{h}{\sigma}\right)^{2}}}\leq\int\abs{\hat{f}\left(\frac{x}{\sigma}\right)-\hat{f}\left(0\right)}\frac{e^{-x^{2}/2}}{\sqrt{2\pi}}\, dx.
\]
Since $f$ is even,
\[
\hat{f}\left(\frac{x}{\sigma}\right)-\hat{f}\left(0\right)=\frac{1}{\sqrt{2\pi}}\int f\left(y\right)\left(e^{i\frac{x}{\sigma}y}-1\right)\, dy=\frac{1}{\sqrt{2\pi}}\int f\left(y\right)\left(\cos\left(\frac{x}{\sigma}y\right)-1\right)\, dy
\]
so that
\[
\abs{\hat{f}\left(\frac{x}{\sigma}\right)-\hat{f}\left(0\right)}\leq\frac{1}{2}\frac{1}{\sqrt{2\pi}}\frac{x^{2}}{\sigma^{2}}\norm{f}_{L_1(x^2dx)}
\]
since $\abs{1-\cos\theta} \leq\frac{1}{2}\theta^{2}$. Therefore
\[
\abs{\left\langle e^{i\frac{h}{\sigma}x}\hat{f}\left(\frac{x}{\sigma}\right),
\frac{e^{-x^{2}/2}}{\sqrt{2\pi}}\right\rangle _{L^{2}\left(dx\right)}-\hat{f}\left(0\right)e^{-\frac{1}{2}\left(\frac{h}{\sigma}\right)^{2}}}
\leq\frac{1}{2}\frac{1}{\sqrt{2\pi}}\frac{1}{\sigma^{2}}\norm{f}_{L_1(x^2dx)}.
\]
\end{proof}
This result shows that the curves
\[
\sigma e^{-\frac{1}{2}\left(\frac{h}{\sigma}\right)^{2}}=const.
\]
are distinguished for problems of the above type. This is made more clear by the following corollary.
\begin{corollary}\label{cor:limit-ratio}
Suppose that $f,g\in\cE$. Assume that $\int f\neq0$ and that
$\lim_{\sigma\to\infty}\,\sigma^{2}\E f\left(h+\sigma Z\right)$ is finite.
Then 
\[
\lim_{\sigma\to\infty}\frac{\sigma^{2}\E f\left(h+\sigma Z\right)}{\int f\left(x\right)\, dx}=\lim_{\sigma\to\infty}\frac{\sigma e^{-\frac{1}{2}\left(\frac{h}{\sigma}\right)^{2}}}{\sqrt{2\pi}}.
\]
Furthermore,
\[
\lim_{\sigma\to\infty}\frac{\E g\left(h+\sigma Z\right)}{\E f\left(h+\sigma Z\right)}=\frac{\int g}{\int f}.
\]
\end{corollary}

\begin{proof}
The assumption that 
\[
\lim_{\sigma\to\infty}\,\sigma^{2}\E f\left(h+\sigma Z\right)=a\in \R
\qquad\text{ 
implies that
 }\qquad
\lim_{\sigma\to\infty}\,\frac{\sigma e^{-\frac{1}{2}\left(\frac{h}{\sigma}\right)^{2}}}{\sqrt{2\pi}}\int f\left(x\right)\, dx=a
\]
by the lemma above. Then by that same lemma,
\[
\sigma^{2}\E g\left(h+\sigma Z\right)  =\frac{\sigma e^{-\frac{1}{2}\left(\frac{h}{\sigma}\right)^{2}}}{\sqrt{2\pi}}\int g\left(x\right)\, dx+o\left(1\right) =a\frac{\int g\left(x\right)\, dx}{\int f\left(x\right)\, dx}+o\left(1\right).\qquad
\]
\end{proof}
We end this section with the following observation.
\begin{corollary}\label{cor:1d-disp-diff-ulb}
Assume that $f,g\in\cE$ and $\int f \neq 0$. Then
\[
\abs{\sigma^{2}\E g\left(h+\sigma Z\right)-\frac{\int g}{\int f}\cdot\sigma^{2}\E f\left(h+\sigma Z\right)}\leq
\frac{1}{2}\frac{1}{\sqrt{2\pi}}\frac{1}{\sigma}\left(\frac{\norm{f}_{L_1(y^2dy)}}{\abs{\int f}}\norm{g}_{L_1(dx)}+\norm{g}_{L_1(y^2dy)}\right).
\]
\end{corollary}
\begin{proof}
Recall that by Lemma \ref{lem:1d-disp}
\[
\abs{\sigma^{2}\E g\left(h+\sigma Z\right) - \frac{\sigma e^{-\frac{1}{2}\left(\frac{h}{\sigma}\right)^{2}}}{\sqrt{2\pi}}\int g\left(x\right)\, dx}
\leq \frac{1}{2}\frac{1}{\sqrt{2\pi}}\frac{1}{\sigma}\norm{g}_{L_1(y^2dy)}.
\]
Similarly, by Lemma \ref{lem:1d-disp} we get that
\[
\abs{\frac{\sigma e^{-\frac{1}{2}(\frac{h}{\sigma})^2}}{\sqrt{2\pi}}-\frac{\sigma^{2}\E f\left(h+\sigma Z\right)}{\int f}}
\leq \frac{1}{2\sqrt{2\pi}\sigma}\frac{\norm{f}_{L_1(y^2dy)}}{\abs{\int f}}.
\]
Note that
\begin{align*}
\abs{\sigma^{2}\E g\left(h+\sigma Z\right)-\frac{\int g}{\int f}\cdot\sigma^{2}\E f\left(h+\sigma Z\right)}
&\leq \abs{\sigma^{2}\E g\left(h+\sigma Z\right) - \frac{\sigma e^{-\frac{1}{2}\left(\frac{h}{\sigma}\right)^{2}}}{\sqrt{2\pi}}\int g\left(x\right)\, dx}\\
&+ \abs{\frac{\sigma e^{-\frac{1}{2}(\frac{h}{\sigma})^2}}{\sqrt{2\pi}}-\frac{\sigma^{2}\E f\left(h+\sigma Z\right)}{\int f}}\cdot\norm{g}_1.
\end{align*}
Combining this with the above two estimates gives the result.
\end{proof}

\subsection{Dispersive estimates in 2-d}\label{sec:2d-disp}
In this subsection, we develop  higher dimensional analogues of the previous results. 
As we will only use the $d=2$ result, and as we think it is illustrative of the heart of the matter, we focus only on this case. 
The setting of this subsection is substantially different from the previous in scope, thus with the eventual application of these results
in mind, we organize this subsection as follows. We begin by stating the main result in full generality. We then focus on a 
special case that clarifies the nature and form of these estimates.  This result will not be used in the proof
of the main results of this paper, but has the added benefit of simplifying the application of these results 
under an additional hypothesis (to be introduced in Sect. 6). 

We begin by stating the main theorem
and its proof. Suppose that we have a family of positive-definite linear operators on $\R^2$ and vectors in $\R^2$ indexed by 
some set $T$, $\{(\Sigma\left(t\right),\mathbf{m}(t))\}_{t\in T}$, where
\[
\Sigma\left(t\right)={\lambda}_{1}\left(t\right){v}_{1}\otimes {v}_{1}\left(t\right)+{\lambda}_{2}\left(t\right)v_{2}\otimes v_{2}\left(t\right)
\]
with $\norm{v_{i}}_2=1$, $v_{1}\perp v_{2}$. Suppose that $\{w_i\}$ is an orthonormal frame in $\R^2$ and let $\nu$ be a positive real number. 

Define the set $\Tbull=T\cup\{\infty\}$, and extend the family of eigenvalues and eigenvectors to $t\in \Tbull$ by
\begin{equation}
\lambda_1(t) =
\begin{cases}
\lambda_1(t) & t\in T\\
\nu & t= \infty
\end{cases}
\qquad
\lambda_2(t) =
\begin{cases}
\lambda_2(t) & t\in T\\
\infty & t= \infty
\end{cases}
\qquad
v_i (t) =
\begin{cases}
v_i(t) & t\in T\\
w_i & t=\infty.
\end{cases}
\end{equation}
For $t\in\Tbull$, let
\begin{equation}\label{eq:a-def}
a(y;t) =
\begin{cases}
\left\langle \mathbf{m}(t),v_{1}(t)\right\rangle +\lambda_{1}^{1/2}y & t\in T\\
 \nu^{1/2}y & t=\infty. 
 \end{cases}
\end{equation}
\begin{definition}
The $\fG_t$-\emph{transform} of $f$ for $t\in\Tbull$ 
is 
\begin{equation}\label{eq:frakg-def}
\fG_t[f](x)=\E\left[f\left(a\left(Z;t\right)v_{1}(t)+xv_{2}(t)\right)\right]
\end{equation}
where the expectation is in $Z$ where $Z$ is a standard Gaussian.
\end{definition}
Given $f\in L_1(\R^2,dx)$ and the frame $\{w_i\}$, we define the \emph{bracket}
of $f$ with respect to $\{w_i\}$ by
\begin{equation} \label{eq:brac-def}
\brac{f}(x)=\int f(x w_1+ y w_2)dy.
\end{equation}

Now we can state our 2-d dispersive estimates.
Define the remainder 
\begin{equation}\label{eq:remainder}
R\left(t\right)=\lambda_{2}^{1/2}I-\frac{e^{-\frac{1}{2}\frac{m_{2}^{2}}{\lambda_{2}}}}{\sqrt{2\pi}}\E\left[\left\langle f\right\rangle \left(\nu^{1/2}Z\right)\right]
\end{equation}
for all $t\in T$,
where 
\[
I=\E f\left(\mathbf{m}+\sqrt{\Sigma}\mathbf{Z}\right),
\quad\text{and}\quad 
m_{i}  =\left\langle \mathbf{m},v_{i}\right\rangle,\  i=1,2,
\]
and where $\mathbf{Z}$ is a standard Gaussian vector. 
Define the errors
\begin{equation}\label{eq:errorsei}
\Delta_{1}(t)=\norm{\fG_t[f]-\fG_\infty[f]}_{L_{1}(dy)}\quad\text{and}\quad \Delta_{2}(t)=\norm{\fG_t[f]-\fG_\infty[f]}_{L_{1}(\abs{y}dy)}.
\end{equation}
Define the family of linear maps $A(t):\R^{2}\to\R^{2}$ by 
\begin{equation}\label{eq:Aop-def}
A(t) = \lambda_1^{1/2}v_1\tensor e_1 + v_2\tensor e_2 
\end{equation}
for $t\in\Tbull$, where $\{e_i\}$ is the standard basis.
Finally, let 
$\cA=\{f: f\in L_1(\norm{x}^2\vee 1dx)\cap L_2(dx),\ f(x)=f(-x)\}$. In the following, we let
$\norm{A}$ denote the Frobenius norm.

\begin{theorem}\label{thm:2d-disp-bound} Let $f\in \cA$. Then for all $t\in T$, 
\begin{equation}
\abs{R}\leq\frac{1}{\lambda_{2}}\frac{1}{2\sqrt{2\pi}}\norm{\fG_\infty[f]}_{L_{1}(\abs{y}^{2}dy)}+\frac{e^{-\frac{1}{2}\frac{m_{2}^{2}}{\lambda_{2}}}}{\sqrt{2\pi}}\Delta_1+\frac{1}{\pi}\frac{1}{\lambda_{2}^{1/2}}\Delta_2.
\end{equation}
Suppose furthermore that $f\in Lip(\R^2)$ 
and has exponential decay
\[
\abs{f\left(x\right)}\leq c_{1}e^{-c_{2}\norm{x}_2}
\]
for some constants $c_{1},c_{2}>0$. Let $c=\max\{\frac{\sqrt{2}}{c_{2}},\left(\frac{\sqrt{2}}{c_{2}}\right)^{2}\}$. Then for every $M\geq 2$ we have the estimates
\[
\Delta_{i}(t)\leq Lip(f)\left[\norm{A(t)-A(\infty)}\left(M^{3}(1+\frac{4}{M}\frac{1-e^{-M^{2}/2}}{\sqrt{2\pi}})\right)+M^2\abs{m_{1}}\right]+4c_{1}c\left[ e^{-\frac{M}{c}}(M+1)+ e^{-\frac{M^{2}}{2}}\right].
\]
Finally we have that $\norm{\fG_\infty[f]}_{L_1(\abs{y}^2dy)}\leq C(f)<\infty$.
\end{theorem}
%
\begin{proof}
The first part is the content of Lemma \ref{lem:2d-disp-bound}. The second part is the content of Corollary \ref{cor:deltai-bounds}. 
The last part comes from the exponential decay of $f$. In particular,
since $f$ has exponential decay, $f=f(uw_1 + vw_2)$ is in the anisotropic Lebesgue space $f\in L_\infty(du)L_1(v^2dv)= B$, 
so that
\begin{align*}
\int \abs{\fG_\infty [f][v]}v^2dv &= \int \abs{\int f(\nu^{1/2}uw_1+ v w_2)e^{-u^2/2}du/\sqrt{2\pi}}v^2 dv\\
&\leq \int \int \abs{f(\nu^{1/2}uw_1+ v w_2)}v^2 dv e^{-u^2/2}du/\sqrt{2\pi} \\
&\leq \norm{f}_{B}.
\end{align*}

\end{proof}
\subsubsection{A Motivating Example: $T=\R_+$} \label{sec:disp-motiv-exam} 
In this setting we consider the specific example where $T=\R_+$.
To this end, fix $\nu > 0$ and for each $\nu$
let $T_\nu=\R_+$. Suppose that for each $\nu$, we have a one-parameter family $\{\Sigma\left(t\right)\}_{t\in T_\nu}$
of symmetric positive-definite matrices 
\[
\Sigma\left(t\right)=\lambda_{1}\left(t\right)v_{1}\otimes v_{1}\left(t\right)+\lambda_{2}\left(t\right)v_{2}\otimes v_{2}\left(t\right)
\]
with $\norm{v_{i}}_2=1$, $v_{1}\perp v_{2}$, and a
one-parameter family $\{\mathbf{m}\left(t\right)\}_{t\in T_\nu}$ of vectors in
$\R^{2}$. 

Assume further that as $t\to+\infty$: 
\begin{itemize}
\item $\left\{ v_{1},v_{2}\right\} \to\left\{ w_1,w_2\right\} $
\item $\lambda_{1}\to\nu$, $\lambda_{2}\to+\infty$
\item $\left\langle \mathbf{m},v_{1}\right\rangle \to0$.
\end{itemize}

Nominally, the goal of higher dimensional dispersive estimates is to study the asymptotics of 
\[
\E f\left(\mathbf{m}(t)+\sqrt{\Sigma(t)}\mathbf{Z}\right),
\]
in this limit where $\mathbf{Z}$ is a standard Gaussian vector. 
In particular, we seek to develop estimates
that depend on the asymptotic spectral properties of $\Sigma_t$. The main technical difficulty that presents itself
is in obtaining estimates that are uniform in $\nu$ as $\nu\rightarrow 0$.

The reader will observe that in 1-d the essential idea is that by inverting, i.e. by moving to Fourier space and rescaling, the large noise limit becomes a small noise limit, so that we are in the classical setting of concentration of measure.
When searching for the analogous 
estimates in higher dimensions, one is tempted to ``invert in every direction'', thereby finding estimates that are in terms of the norm
of $\Sigma^{-1/2}$. This will inevitably lead to issues as $\nu\rightarrow 0$. Put simply, if one seeks mixed large noise-small noise limits, one cannot simply work in physical or Fourier space alone.
The main idea behind these results is that one should only
invert in the directions that disperse and use the regularity of $f$ to obtain estimates that are uniform 
in the remaining variables.


Now we prove the analog of Lemma \ref{lem:softdisp}. These results will not be used in the subsequent. Instead they are motivational. The reader will observe that the decision not to invert in both directions is reflected by the iterated Gaussian-bracket structure of the integral in Proposition \ref{prop:2dsoftdisp}, and by the appearance of $\nu$ in the argument of the integrand.  We begin by observing the following bounds.
\begin{lemma}\label{lem:frakG-bounds}
Under the assumptions of Section \ref{sec:disp-motiv-exam}, the operators $\fG_t: L_1(\R^2,dx)\rightarrow L_1(\R,dx)$ satisfy
\[
\norm{\fG_t}\leq \frac{1}{\sqrt{2\pi \lambda_1(t)}}.
\]
In particular since $\lambda_1\rightarrow \nu >0$, they are uniformly bounded in $t$. Furthermore,
$\fG_t\rightarrow\fG_\infty$ in the strong operator topology.
\end{lemma}
\begin{remark}\label{rem:LinfL1}
Similar estimates are true for $L_1(\R^2,(\norm{x}^k\vee 1)dx)$, and $L_2$ as well. Note that if 
$f(u v_1 + w v_2)\in L_{\infty}^uL_1^w$ we get that $\fG_t$ has $L^u_\infty L^w_1\rightarrow L_1$ norm at most 1. 
\end{remark}

\begin{proposition} \label{prop:2dsoftdisp}
Let $f\in L_{1}\left(\R^{2},dx\right)\cap L_2\left(\R^2,dx\right)$. We have that
\[
\limsup\,\lambda_{2}^{1/2}\E f\left(\mathbf{m}+\sqrt{\Sigma}\mathbf{Z}\right)=\frac{e^{-\frac{1}{2}\liminf\frac{m_2^2}{\lambda_{2}}}}{\sqrt{2\pi}}\E\left[\left\langle f\right\rangle \left(\nu^{1/2}Z\right)\right]
\]
where $\brac{f}$ is defined in (\ref{eq:brac-def}) and $Z$ is a standard normal random variable.
\end{proposition}
\begin{remark}
The reader will observe that there is a more elementary proof of this result. The extra effort will be rewarded in the proceeding as it will allow us to read off the proof of Theorem \ref{thm:2d-disp-bound} with ease, as in the 1-d case.
\end{remark}
\begin{proof}
Fix $t\in T$. Since
$\left\{ v_{1},v_{2}\right\} $ is an eigenbasis for $\Sigma$ and
an orthonormal basis for $\R^{2}$, we have that
\[
\mathbf{m}+\sqrt{\Sigma}\mathbf{Z}=\left(\left\langle \mathbf{m},v_{1}\right\rangle +\lambda_{1}^{1/2}\left\langle \mathbf{Z},v_{1}\right\rangle \right)v_{1}+\left(\left\langle \mathbf{m},v_{2}\right\rangle +\lambda_{2}^{1/2}\left\langle \mathbf{Z},v_{2}\right\rangle \right)v_{2},
\]
and
\[
I=\int_{\R^{2}}f\left(\left(\left\langle \mathbf{m},v_{1}\right\rangle +\lambda_{1}^{1/2}y_{1}\right)v_{1}+\left(\left\langle \mathbf{m},v_{2}\right\rangle +\lambda_{2}^{1/2}y_{2}\right)v_{2}\right)e^{-\frac{1}{2}\abs{y}^{2}}\,\frac{dy}{2\pi}.
\]
Thus we can write $I$ as an iterated integral, 
\[
I= \innerp{\fG_t[f](\innerp{\bm,v_2}+\lambda_2^{1/2}y_2),\frac{e^{-\frac{1}{2}\abs{y_2}^2}}{\sqrt{2\pi}}}_{L_2 (dy_2)}.
\]
Now we can apply 1-d Fourier methods. Introduce the dual variable $k_{2}\leftrightarrow y_{2}$
and apply Plancherel to get that
\begin{equation}\label{eq:I-in-fourier}
\lambda_2^{1/2}I 
= \innerp{e^{ik_2\frac{\innerp{\bm,v_2}}{\lambda_2^{1/2}}}\widehat{\fG_t[f]}\left(\frac{k_2}{\lambda_2^{1/2}}\right),
\frac{e^{-\frac{1}{2} k_2^2}}{\sqrt{2\pi}}}.
\end{equation}
We want to take $t\to+\infty$. 

By assumption, we have that 
$a\left(y_{1};t\right)\to\nu^{1/2}y_{1}$
pointwise. We also have that
\[
\fG_t[f](y_{2})\to\fG_\infty[f](y_2)=\E\left[f\left(\nu^{1/2}Zw_1+y_{2}w_2\right)\right]
\]
in $L_1\left(dy_{2}\right)$ by Lemma \ref{lem:frakG-bounds}.
As $\wedge:L_{1}\left(\R^{1}\right)\to C_0\left(\R^{1}\right)$
continuously, it follows that $\widehat{\fG_t[f]}\to\widehat{\fG_\infty[f]}$ uniformly.
Therefore
\[
\widehat{\fG_t[f]}\left(\frac{k_{2}}{\lambda_{2}^{1/2}}\right)\to\widehat{\fG_\infty[f]}\left(0\right)
\]
pointwise. By the bounded convergence theorem, we then get that
\begin{align*}
\limsup\,\lambda_{2}^{1/2}I  =\limsup\,\left\langle e^{ik_{2}\frac{\left\langle \mathbf{m},v_{2}\right\rangle }{\lambda_{2}^{1/2}}}\widehat{\fG_\infty[f]}\left(0\right),\frac{e^{-\frac{1}{2}\abs{k_{2}}^{2}}}{\sqrt{2\pi}}\right\rangle _{L^{2}\left(dk_{2}\right)} =e^{-\frac{1}{2}\liminf\left(\frac{\left\langle \mathbf{m},v_{2}\right\rangle }{\lambda_{2}^{1/2}}\right)^{2}}\widehat{\fG_\infty[f]}\left(0\right).
\end{align*}
Finally,
\[
\widehat{\fG_\infty[f]}\left(0\right)=\frac{1}{\sqrt{2\pi}}\int \fG_\infty[f]\left(y_{2}\right)\, dy_{2}=\frac{1}{\sqrt{2\pi}}\E\left[\left\langle f\right\rangle \left(\nu^{1/2}Z\right)\right].
\]
\end{proof}

\subsubsection{Proof of Main Estimate in General Setting}

\begin{lemma}\label{lem:2d-disp-bound} Let $f\in \cA$. Then for all $t\in T$, 
\begin{equation}
\abs{R}\leq\frac{1}{\lambda_{2}}\frac{1}{2\sqrt{2\pi}}\norm{\fG_\infty[f]}_{L_{1}(\abs{y_{2}}^{2}dy_{2})}+\frac{e^{-\frac{1}{2}\frac{m_{2}^{2}}{\lambda_{2}}}}{\sqrt{2\pi}}\Delta_1+\frac{1}{\pi}\frac{1}{\lambda_{2}^{1/2}}\Delta_2.
\end{equation}
\end{lemma}
\begin{remark}\label{rem:disp-bound}
Since the operators are converging strongly, we expect the last two terms to vanish as $t\rightarrow\infty$, though we
expect the rates of convergence to depend on the regularity of $f$. Note, that these expressions will explode
as $\nu \rightarrow 0$ in general. Upon adding mild regularity requirements on $f$,
however, we see that these expressions remain bounded as $\nu\rightarrow 0$. 
Furthermore the righthand side immediately reduces to the 1-d estimates from Lemma \ref{lem:1d-disp}
when
\begin{itemize}
\item $v_{2}=w_2$ (so that $v_{1}=w_1$)
\item $\lambda_{1}=\nu=0$
\item $m_1=0$.
\end{itemize}
Indeed, under these stronger assumptions you get $\fG_t=\fG_\infty= R$ where $R$ is the restriction
operator in the direction of $w_2$, and the lemma
becomes
\[
\abs{\lambda_{2}^{1/2}I-\frac{e^{-\frac{1}{2}\frac{m_{2}^{2}}{\lambda_{2}}}}{\sqrt{2\pi}}\left\langle f\right\rangle \left(0\right)}\leq\frac{1}{\lambda_{2}}\frac{1}{2\sqrt{2\pi}}\norm{\fG_\infty[f]}_{L_{1}\left(\abs{y_{2}}^{2}dy_{2}\right)}
\]
which is exactly the 1-d dispersive bound.
\end{remark}
\begin{proof}
 Fix $t\in T$ and let $\lambda_i$, $m_i$ be their values at $t$. Recall $R$ from \eqref{eq:remainder}. As in the proof of Proposition \ref{prop:2dsoftdisp} (specifically \eqref{eq:I-in-fourier}), observe that
\begin{align*}
R & =
\innerp{e^{ik_{2}\frac{m_{2}}{\lambda_{2}^{1/2}}}\left(\widehat{\fG_t[f]}\left(\frac{k_{2}}{\lambda_{2}^{1/2}}\right)-\widehat{\fG_\infty[f]}(0)\right),
\frac{e^{-\frac{1}{2}\abs{k_{2}}^{2}}}{\sqrt{2\pi}}}_{L^{2}\left(dk_{2}\right)}\\
 & =\innerp{ e^{ik_{2}\frac{m_{2}}{\lambda_{2}^{1/2}}}\left(\widehat{\fG_t[f]}\left(\frac{k_{2}}{\lambda_{2}^{1/2}}\right)-\widehat{\fG_t[f]}\left(0\right)\right),\frac{e^{-\frac{1}{2}\abs{k_{2}}^{2}}}{\sqrt{2\pi}}}+ \innerp{e^{ik_{2}\frac{m_{2}}{\lambda_{2}^{1/2}}}\left(\widehat{\fG_t[f]}(0)-\widehat{\fG_\infty[f]}\left(0\right)\right),\frac{e^{-\frac{1}{2}\abs{k_{2}}^{2}}}{\sqrt{2\pi}}}\\
 & =E_{1}+E_{2}.
\end{align*}

For $E_{1}$, write
\begin{align*}
&\widehat{\fG_t[f]}\left(\frac{k_{2}}{\lambda_{2}^{1/2}}\right)-\widehat{\fG_t[f]}(0) =\frac{1}{\sqrt{2\pi}}\int \fG_t[f](y_{2})\left(e^{-i\frac{k_{2}}{\lambda_{2}^{1/2}}y_{2}}-1\right)dy_{2}\\
 & =\frac{1}{\sqrt{2\pi}}\int\left(\fG_t[f](y_{2})-\fG_\infty[f](y_{2})\right)\left(e^{-i\frac{k_{2}}{\lambda_{2}^{1/2}}y_{2}}-1\right)dy_{2}+\frac{1}{\sqrt{2\pi}}\int \fG_\infty[f](y_{2})\left(e^{-i\frac{k_{2}}{\lambda_{2}^{1/2}}y_{2}}-1\right)dy_{2}\\
 & =\delta_{1}+\delta_{2}.
\end{align*}
Now
\begin{align*}
\abs{\delta_{1}} 
  \leq\frac{1}{\sqrt{2\pi}}\frac{\abs{k_{2}}}{\lambda_{2}^{1/2}}\norm{\fG_t[f](y_{2})-\fG_\infty[f](y_{2})}_{L_{1}\left(\abs{y_{2}}dy_{2}\right)},
\end{align*}
and
\begin{align*}
\delta_{2} & =\frac{1}{\sqrt{2\pi}}\int \fG_\infty[f](y_{2})\left(\cos\left(\frac{k_{2}}{\lambda_{2}^{1/2}}y_{2}\right)-1\right)\, dy_{2}
\end{align*}
because $\fG_\infty[f](y_{2})=\fG_\infty[f](-y_{2})$ by the assumptions on $f$. Therefore
\[
\abs{\delta_{2}}\leq\frac{1}{\lambda_{2}}\frac{\abs{k_{2}}^{2}}{2\sqrt{2\pi}}\norm{\fG_\infty[f]}_{L_{1}\left(\abs{y_{2}}^{2}dy_{2}\right)}.
\]
Combining the above and using the triangle inequality yields
\begin{align*}
\abs{E_{1}} &\leq \innerp{ \abs{\delta_{1}}+\abs{\delta_{2}},\frac{e^{-\frac{1}{2}\abs{k_{2}}^{2}}}{\sqrt{2\pi}}}
 \leq \frac{1}{\pi}\frac{1}{\lambda_{2}^{1/2}}\norm{\fG_t[f]-\fG_\infty[f]}_{L_{1}(\abs{y}dy)}+\frac{1}{\lambda_{2}}\frac{1}{2\sqrt{2\pi}}\norm{\fG_\infty[f]}_{L_{1}(\abs{y_{2}}^{2}dy_{2})}.
\end{align*}
(Here we used that the first absolute moment of a standard Gaussian is $\sqrt{2/\pi}$.)

For $E_{2}$, write
\[
\abs{E_{2}}  =e^{-\frac{1}{2}\frac{m_{2}^{2}}{\lambda_{2}}}\abs{\widehat{\fG_t[f]}(0)-\widehat{\fG_\infty[f]}(0)}
  \leq \frac{e^{-\frac{1}{2}\frac{m_{2}^{2}}{\lambda_{2}}}}{\sqrt{2\pi}}\norm{\fG_t[f]-\fG_\infty[f]}_{L_{1}(dy)}.
\]
Adding the bounds for $E_{1}$, $E_{2}$ and rearranging gives the
claim.
\end{proof}

\subsubsection{Rates of convergence from energy estimates}
In this subsection we will be concerned with computing estimates on the errors $\Delta_i$. For ease of reading, and with applications in mind, we assume in this section that 
\[
w_1 = \frac{(-1,1)}{\sqrt{2}} \quad \text{and} \quad w_2 = \frac{(1,1)}{\sqrt{2}}.
\]
The reader will observe that as the inequalities in Theorem
\ref{thm:2d-disp-bound} are rotationally invariant, we can assume this without loss of generality. In this subsection, we assume that $t\in T$.

In this subsection, we assume that $f\in Lip\left(\R^{2}\right)$ and has exponential decay
\[
\abs{f\left(x\right)}\leq c_{1}e^{-c_{2}\norm{x}_2}
\]
for some constants $c_{1},c_{2}>0$.  Let
\begin{align*}
d\mu_{1}  =d\gamma\left(y_{1}\right)\otimes dy_{2}\quad\text{ and }\quad
d\mu_{2}  =d\gamma\left(y_{1}\right)\otimes\abs{y_{2}}dy_{2}.
\end{align*}
where $d\gamma(x)$ is the standard Gaussian measure on $\R$ and define the measures
\begin{align*}
\kappa_{i}\left(K\right)  =\int_{K}\norm{y}_2\, d\mu_{i}\left(y\right)\qquad\text{and}\qquad
\upsilon_{i}\left(K;t\right)  =\int_{\R^{2}\backslash K}e^{-c_{2}\norm{A(t)y+m_{1}v_{1}}_2}+e^{-c_{2}\norm{A(\infty)y}_2}\, d\mu_{i}\left(y\right)
\end{align*}
for $K\in\mathcal{B}$. 

\begin{lemma}
For $f$ as above and $t\in T$,
\[
\Delta_{i}(t)\leq Lip\left(f\right)\left(\norm{A(t)-A(\infty)}\cdot\kappa_{i}\left(K\right)+\abs{m_{1}(t)}\cdot\mu_{i}\left(K\right)\right)+c_{1}\upsilon_{i}\left(K;t\right),\quad i=1,2.
\]
\end{lemma}
\begin{proof}
Let $\lambda_i$, $m_i$ be their values at $t\in T$. 
We have that
\begin{align*}
\Delta_{1} & =\norm{\E\left[f\left(a\left(Z;t\right)v_{1}+y_{2}v_{2}\right)-f\left(\nu^{1/2}Zw_1+y_{2}w_2\right)\right]}_{L_{1}\left(dy_{2}\right)}\\
 & =\int\abs{\int\left(f\left(\left(m_{1}+\lambda_{1}^{1/2}y_{1}\right)v_{1}+y_{2}v_{2}\right)-f\left(\nu^{1/2}y_{1}w_1+y_{2}w_2\right)\right)\frac{e^{-y_{1}^{2}/2}}{\sqrt{2\pi}}\, dy_{1}}\, dy_{2}
\end{align*}
and
\begin{align*}
\Delta_{2} & =\norm{\E\left[f\left(a\left(Z;t\right)v_{1}+y_{2}v_{2}\right)-f\left(\nu^{1/2}Zw_1+y_{2}w_2\right)\right]}_{L_{1}\left(\abs{y_{2}}dy_{2}\right)}\\
 & =\int\abs{\int\left(f\left(\left(m_{1}+\lambda_{1}^{1/2}y_{1}\right)v_{1}+y_{2}v_{2}\right)-f\left(\nu^{1/2}y_{1}w_1+y_{2}w_2\right)\right)\frac{e^{-y_{1}^{2}/2}}{\sqrt{2\pi}}\, dy_{1}}\abs{y_{2}}\, dy_{2}.
\end{align*}
The errors satisfy
\begin{align*}
\Delta_{i} & \leq\int\abs{f\left(A(t)y+m_{1}v_{1}\right)-f\left(A(\infty)y\right)}\, d\mu_{i}\left(y\right),\quad i=1,2.
\end{align*}
Since $f$ is Lipschitz, 
\[
\abs{f\left(A(t)y+m_{1}v_{1}\right)-f\left(A(\infty)y\right)}\leq Lip\left(f\right)\norm{\left(A(t)-A(\infty)\right)y+m_{1}v_{1}}\leq Lip\left(f\right)\left(\norm{A(t)-A(\infty)}\cdot\norm{y}+\abs{m_{1}}\right)
\]
so that
\begin{align*}
\int_{K}\abs{f\left(A(t)y+m_{1}v_{1}\right)-f\left(A(\infty)y\right)}\, d\mu_{i}\left(y\right) & \leq Lip\left(f\right)\left(\norm{A(t)-A(\infty)}\int_{K}\norm{y}\, d\mu_{i}\left(y\right)+\abs{m_{1}}\mu_{i}\left(K\right)\right)\\
 & =Lip\left(f\right)\left(\norm{A(t)-A(\infty)}\cdot\kappa_{i}\left(K\right)+\abs{m_{1}}\cdot\mu_{i}\left(K\right)\right).
\end{align*}
By the exponential decay of $f$, we get that
\begin{align*}
\int_{\R^{2}\backslash K}\abs{f\left(A(t)y+m_{1}v_{1}\right)-f\left(A(\infty)y\right)}\, d\mu_{i}\left(y\right) & \leq\int_{\R^{2}\backslash K}\abs{f\left(A(t)y+m_{1}v_{1}\right)}+\abs{f\left(A(\infty)y\right)}\, d\mu_{i}\left(y\right)\\
 & \leq c_{1}\int_{\R^{2}\backslash K}e^{-c_{2}\norm{A(t)y+m_{1}v_{1}}}+e^{-c_{2}\norm{A(\infty)y}}\, d\mu_{i}\left(y\right)
  =c_{1}\upsilon_{i}\left(K;t\right).
\end{align*}
as desired.
\end{proof}
The next step is to optimize over $K$. 
\begin{lemma}
(The $\upsilon_{i}$ estimates.)  Let $c=\max\{\frac{\sqrt{2}}{c_{2}},\left(\frac{\sqrt{2}}{c_{2}}\right)^{2}\}$.
Then if $K=[-M,M]^{2}$ and $M\geq 1$, we have that
\[
\upsilon_{i}(K(M);t)\leq4\left[c e^{-\frac{M}{c}}\left(M+1\right)
+c e^{-\frac{M^{2}}{2}}\right].
\]
\end{lemma}
\begin{proof}
Write 
\[
d\mu_{i}=d\gamma(y_{1})\tensor d\ell_{i}(y_{2})
\]
where $d\ell_{1}(y)=dy$ and $d\ell_{2}=\abs{y}dy$, and recall that
\[
\upsilon_{i}(K)=\int_{K^{c}}e^{-c_{2}\norm{A(t)y+m_{1}v_{1}}_{2}}+e^{-c_{2}\norm{A(\infty)y}_{2}}d\mu_{i}.
\]
Note that 
\begin{align*}
\norm{A(t)y+m_{1}v_{1}}_{2}  \geq\frac{1}{\sqrt{2}}\left(\abs{m_{1}+\lambda_{1}^{1/2}y_{1}}+\abs{y_{2}}\right)\qquad
\norm{A(\infty) y}_{2}  \geq\frac{1}{\sqrt{2}}\left(\nu^{1/2}\abs{y_{1}}+\abs{y_{2}}\right)
\end{align*}
by the $\ell_{1}-\ell_{2}$ norm inequality, so that 
\[
\upsilon_{i}(K)\leq\int_{K^{c}}e^{-\frac{c_{2}}{\sqrt{2}}\left(\abs{\lambda_{1}^{1/2}y_{1}+m_{1}}+\abs{y_{2}}\right)}+e^{-\frac{c_{2}}{\sqrt{2}}\left(\nu^{1/2}\abs{y_{1}}+\abs{y_{2}}\right)}d\mu_{i}
\leq 2\int_{K^{c}}e^{-\frac{c_{2}}{\sqrt{2}}\abs{y_{2}}}d\mu_{i}.
\]

By Gaussian concentration, 
\begin{align*}
\int_{K^{c}}e^{-\frac{c_{2}}{\sqrt{2}}\abs{y_{2}}}d\mu_{i} & \leq2\left[\int_{\R}\int_{M}^{\infty}e^{-\frac{c_{2}}{\sqrt{2}}\abs{y_{2}}}d\ell_{i}(y_{2})d\gamma(y_{1})+\int_{\R}\int_{M}^{\infty}e^{-\frac{c_{2}}{\sqrt{2}}\abs{y_{2}}}d\gamma(y_{1})d\ell_{i}(y_{2})\right]\\
 & \leq2\left[1\cdot\int_{M}^{\infty}e^{-\frac{c_{2}}{\sqrt{2}}w}d\ell_{i}(w)+e^{-\frac{M^{2}}{2}}\int_{0}^{\infty}e^{-\frac{c_{2}}{\sqrt{2}}w}d\ell_{i}(w)\right]=I.
\end{align*}
Note that we have the inequalities
\begin{align*}
\int_{0}^{\infty}e^{-\frac{c_{2}}{\sqrt{2}}w}d\ell_{i}(w)\leq c \qquad \text{ and }\qquad
\int_{M}^{\infty}e^{-\frac{c_{2}}{\sqrt{2}}w}d\ell_{i}(w) 
  \leq c e^{-\frac{M}{c}}\left(M+1\right)
\end{align*}
so that 
\begin{align*}
I & \leq2\left[ce^{-\frac{M}{c}}\left(M+1\right)+c e^{-\frac{M^{2}}{2}}\right].
\end{align*}
Note that this bound does not depend on $\nu$, so that 
\[
\upsilon_{i}(K)\leq4\left[ce^{-\frac{M}{c}}\left(M+1\right)+ c e^{-\frac{M^{2}}{2}}\right].\qquad 
\]
\end{proof}
Note the following elementary estimates which follow from Gaussian concentration.
\begin{lemma}
(The $\kappa_{i}$ estimates.) If $K\left(M\right)=\left[-M,M\right]^{2}$, then
\begin{align*}
\kappa_{1}\left(K\left(M\right)\right)  &\leq M^{2}\cdot\left(1+\frac{4}{M}\frac{1-e^{-M^{2}/2}}{\sqrt{2\pi}}\right)\\
\kappa_{2}\left(K\left(M\right)\right)  &\leq M^{3}\cdot\left(\frac{2}{3}+\frac{2}{M}\frac{1-e^{-M^{2}/2}}{\sqrt{2\pi}}\right).
\end{align*}
\end{lemma}
As a result, we get:
\begin{corollary}\label{cor:deltai-bounds}
Let $c=\max\{\frac{\sqrt{2}}{c_{2}},\left(\frac{\sqrt{2}}{c_{2}}\right)^{2}\}$ and assume that $M\geq 2$.
Then we have the estimates
\[
\Delta_{i}(t)\leq Lip(f)\left[\norm{A(t)-A(\infty)}\left(M^{3}(1+\frac{4}{M}\frac{1-e^{-M^{2}/2}}{\sqrt{2\pi}})\right)+M^2\abs{m_{1}}\right]+4c_{1}\left[ce^{-\frac{M}{c}}(M+1)+c e^{-\frac{M^{2}}{2}}\right].
\]
\end{corollary}
\begin{proof}
By the above lemmas from this subsection, we see that we only need to study $\mu_i(K)$. The inequality follows after noting that 
\[
\mu_i([-M,M])\leq M^2.\qquad
\]
\end{proof}

\section{The 2/3rd AT line}\label{sec:23-arg}
In this section, we study the 2/3-AT line argument outlined in Sect.\ \ref{sec:23-ATline}.
We begin by a reduction of the problem. Then, by using the 1-d techniques from
Sect.\ \ref{sec:disp-gauss} we prove the result. Recall the definitions of $q_*$ and $\alpha$ in (\ref{eq:q-alpha-def}), and the definition
of the sets $AT$ and $RS$ in (\ref{eq:AT}). We will be following the notation introduced in Sect. \ref{sec:RSprelims}. We remind
the reader here that in this section and in the following, \textbf{we 
will always take $\mu=\delta_{q_*}$}.

\begin{lemma}\label{lem:23-reduction}
If $(\beta,h)\in AT$ and 
\[
\xi''(y)(1-q_*)\leq 1 \qquad \forall y\geq q_*
\]
then $(\beta,h)\in  RS$.
\end{lemma}
\begin{proof}
Suppose that $(\beta,h)\in AT$. Recall  from Lemma \ref{lem:strats} that it suffices to show that $g\leq 0$ on $[q_*,1]$.
Using \eqref{eq:g'} and the formulas for $u$ in Fact \ref{fact:explicitsolns},  we have that
\begin{align*}
g'(y)&=\xi''(y)\E_h \sech^4(X_y)-1\leq \xi''(y)\E_h\sech^2(X_y)-1 \\
&=\xi''(y)\E_h(1-\tanh^2(X_y))-1= \xi''(y)(1-y)-\xi''(y)g(y)-1.
\end{align*}
Thus, $g$ satisfies the differential inequality
\[
g'(y)+\xi''(y)g(y)\leq \xi''(y)(1-y)-1\leq\xi''(y)(1-q_*)-1\leq 0.
\]
Since $g(q_*)=0$, a comparison argument (as in the proof of Lemma \ref{lem:sk-3/2}) shows that $g(y)\leq0$ for all $y \geq q_*$.
\end{proof}

We will need the following lemmas which will be used frequently in the subsequent. The first lemma concerns a lower bound on $q_*$.
\begin{lemma}\label{lem:lbq} 
We have the following lower bounds on $q_*$:
\begin{enumerate}
\item For all models $\xi_0$, 
\[
q_* \geq \frac{1}{2}\tanh^2h.
\]
\item If $\xi''_0(0)=2\beta_2^2\neq 0$ then for $\beta>0$,
\[
q_*\geq 1-\frac{\sqrt{\alpha}}{\sqrt{2}\beta\beta_2}.
\]
\end{enumerate}
\end{lemma}
\begin{remark}
Here we see the main difference between the setting $\xi''(0)>0$ and the setting $\xi''(0)\geq0$ for our arguments. This lemma is the main reason for the
assumption $h\geq h_0>0$ in the following. 
When focusing on the case when $\xi''(0)>0$, 
one can take $h_0=0$ if one assumes that $\beta$ is sufficiently large and the following analysis will hold \emph{mutatis mutandis}.
\end{remark}
\begin{proof}
The first claim follows from the definition of $q_*$ in (\ref{eq:q-alpha-def}) and the fact that $\tanh^2x$ is non-decreasing for $x\geq 0$.
To prove the second claim, observe that by Jensen's inequality
\[
\xi''(q_*)(1-q_*)^2\leq \xi''(q_*)\E\sech^4\left(\sqrt{\xi'(q_*)}Z+h\right)=\alpha 
\]
and therefore
\[
1-q_* \leq \frac{\sqrt{\alpha}}{\sqrt{2}\beta\beta_2}.\qquad
\]
\end{proof}

The next lemma and its corollary use the techniques from Sect.\ \ref{sec:disp-gauss} to estimate $q_*$ in terms of $\alpha$.
Recall that $\Lambda_0=(\pi^2-3)/(6\sqrt{2\pi})$.
\begin{lemma}
Let $f(x)=\sech^{4}(x)$ and $g(x)=\sech^{2}(x)$. We have the inequality
\[
\abs{\sigma^{2}\E g\left(h+\sigma z\right)-\frac{3}{2}\sigma^{2}\E f\left(h+\sigma z\right)}\leq\frac{\Lambda_0}{\sigma}.
\]
\end{lemma}
\begin{proof}
Note that \cite{IntegralTable}
\begin{align*}
\int f=\frac{4}{3}  \qquad\int g=2 \qquad
\int f\left(y\right)y^{2}=\frac{1}{9}\left(\pi^{2}-6\right)  \qquad\int g\left(y\right)y^{2}=\frac{\pi^{2}}{6}\qquad
\frac{\int g}{\int f}=\frac{3}{2}
\end{align*}
and 
\begin{equation}\label{eq:Lambda_0}
\frac{1}{2}\frac{1}{\sqrt{2\pi}}\left(\frac{\int f\left(y\right)y^{2}\, dy}{\int f\left(x\right)\, dx}\int g\left(x\right)\, dx+\int g\left(y\right)y^{2}\, dy\right)=\frac{1}{2}\frac{\pi^{2}-3}{3\sqrt{2\pi}}
=\Lambda_0.
\end{equation}
The result then follows by Corollary \ref{cor:1d-disp-diff-ulb}.
\end{proof} 

\begin{corollary}\label{cor:q-lim}
If $h>0$, we have the estimate
\[
\abs{\xi''(q_*)(1-q_*)- \frac{3}{2}\alpha}\leq \frac{\Lambda_0\xi''_0(q_*)}{\beta(\xi'_0(q_*))^{3/2}}.
\]
\end{corollary}

\paragraph{Hypothesis H} \label{par:hypH} Given certain additional assumptions on the structure of the level sets $\alpha(\beta,h)=$ const., we can compute the rescaled limit of $1-q_*$ as $\beta\to \infty$. 
We call \textbf{hypothesis H} the assumption that the level sets $\{(\beta,h) : \alpha(\beta,h)=const.\}$ are unbounded in $\beta$. We note that hypothesis H is not used in the proof of the main results of this paper, but only used for motivational calculations.
With this, the previous corollary immediately implies the following result.
\begin{corollary} Assume hypothesis H holds. Let $\{(\beta_n,h_n)\}$ be a sequence belonging to the level set $\{(\beta,h) : \alpha(\beta,h) = \tilde{\alpha}\}$ such that $\beta_n\to\infty$ and $h_n\geq h_0>0$. Then,
\[
\lim_{n\to\infty} \, \xi''(q_*)(1-q_*) = \frac{3}{2}\tilde{\alpha}.
\]
\end{corollary}

We conclude this section by proving the following theorem:
\begin{theorem*}\textbf{\ref{prop:23-AT}}
 For any model $\xi_0$, 
\begin{equation}
\{ h>0,\ \alpha \leq \frac{2}{3}\frac{\xi''_0(q_*)}{\xi''_0(1)}\left(1-\frac{\Lambda_0\xi''_0(1)}{\beta(\xi'_0(q_*))^{3/2}}\right)\} \subset RS.
\end{equation}
\end{theorem*}
\begin{proof}
Observe first that $(\beta,h)\in AT$.
Since $\xi''$ is increasing, we see from Lemma \ref{lem:23-reduction} that if $(\beta,h)\in AT$ satisfes
\[
 \xi''(1)(1-q_*)\leq 1
\]
then $(\beta,h)\in RS$. We also see from Corollary \ref{cor:q-lim} that
\[
 \xi''(1)(1-q_*) \leq \frac{\xi''(1)}{\xi''(q_*)} \left (\frac{3}{2}\alpha + \frac{\Lambda_0 \xi''_0(q_*)}{\beta \left(\xi'_0(q_*)\right)^{3/2}}\right).
\]

Combining these gives the result.
\end{proof}

 \section{The long time argument}\label{sec:long-time}
 
In this section, we show that the AT line conjecture is true for $\beta$ large enough. In particular, we prove Theorem \ref{thm:long-time}. We will be following the notation introduced in Sect. \ref{sec:RSprelims}. We remind
the reader here that in this section, \textbf{we 
will always take $\mu=\delta_{q_*}$}.

Observe that by
Lemma \ref{lem:lbq}, if we define $q_0=q_0(h_0)$ by
\begin{equation}\label{eq:q_0}
q_0 =
\frac{1}{2}\tanh^2(h_0),
\end{equation}
it follows that
$q_* \geq q_0$ for $h\geq h_0$. 
The reader will observe that in the following, if a model 
satisfies $\xi''_0(0)>0$, then  by Lemma \ref{lem:lbq}, $q_*$
has a lower bound that depends only on $\beta_2$ and $\beta$ in
the region $\alpha\leq 1$. For such models, one 
can take $h_0=0$ in the following, provided one makes the changes
described at the end of the proof of Theorem \ref{thm:long-time}.

We begin by stating the main technical lemma and use this to prove the theorem. We then end with the proof of said lemma.

\begin{lemma}\label{lem:longtime-main-lem}
For any model $\xi_0$ and any $\alpha_0,h_0>0$, there exist constants $c,C,\beta_0>0$ depending only on $\xi_0,h_0,\alpha_0$
such that for all $\beta,h$ satisfying $\beta\geq\beta_0$, $h\geq h_0$, and $\alpha\in (\alpha_0,1]$, we have that
\[
\E_h\left(4 \sech^4(X_t)-6\sech^6(X_t)\right)\leq -\frac{c}{\beta^2}+\frac{C\log(\beta)^3}{\beta^{5/2}},\quad t\geq q_*.
\]
\end{lemma}

\begin{theorem*}\emph{\textbf{\ref{thm:long-time}}}
For any model $\xi_0$ and positive external field $h_0>0$, there is a $\beta_u$ such that for $\beta\geq\beta_u$ and $h\geq h_0$, the region $\alpha\leq 1$ is in the RS phase. That is, 
\[
AT\cap \{ \beta \geq \beta_u,\ h\geq h_0\} \subset RS.
\]
Furthermore, if $\xi_0''(0)> 0$, then we can take $h_0=0$.
\end{theorem*}
\begin{proof}
%
By Proposition \ref{prop:23-AT},
we see that for $\beta$ sufficiently large we may assume that $\alpha>\alpha_0$ for some $\alpha_0>0$. Similarly, we may assume that the right hand side of the bound in Lemma \ref{lem:longtime-main-lem} is negative.

Now recall that by Lemma \ref{lem:strats}, it suffices to prove that $g'\leq0$ on $[q_*,1]$ to conclude that $(\beta,h)\in RS$. We observe by \eqref{eq:g'}, It\^o's lemma (Sect. \ref{sec:AC-SDE}), and Fact \ref{fact:explicitsolns} that for $y\geq q_*$,
\begin{align*}
g'(y)
&= \xi''(y)\E_hu_{xx}^2(q_*,X_{q_*})-1 + \xi''(y)\int_{q_*}^y\xi''(t)\E_h\left[ 4 \sech^4(X_t)-6\sech^6(X_t)\right]\\
&= \frac{\xi''(y)}{\xi''(q_*)}(\alpha-1)+\frac{\xi''(y)-\xi''(q_*)}{\xi''(q_*)}+ \xi''(y)\int_{q_*}^y\xi''(t)\E_h\left[ 4 \sech^4(X_t)-6\sech^6(X_t)\right]\\
&\leq \frac{\xi'''_0(1)}{\xi''_0(q_0)}(y-q_*)+ \xi''(y)\int_{q_*}^y\xi''(t)\E_h\left[ 4 \sech^4(X_t)-6\sech^6(X_t)\right] = (*)
\end{align*}
where the inequality follows by using the fact that $\alpha \leq1$ on the first term and the mean value theorem, the fact that $q_*\geq q_0$, and the fact that $\xi$ and all of its derivatives are monotone on the second term. Lemma \ref{lem:longtime-main-lem}
then implies that for $\beta$ large enough, 
\[
(*) \leq \frac{\xi'''_0(1)}{\xi''_0(q_0)}(y-q_*)+\xi''(q_*)(\xi'(y)-\xi'(q_*))\left(-\frac{c}{\beta^2}+C\frac{1}{\beta^{5/2-\delta}}\right) =(**).
\]
where $c$ and $C$ are independent of $\beta$. By a similar mean value and monotonicity argument, observe that
\[
\xi''(q_*)(\xi'(y)-\xi'(q_*))\geq \beta^4 \xi_0''(q_0)^2(y-q_*),
\]
from which it follows that
\[
(**) \leq \frac{\xi'''_0(1)}{\xi''_0(q_0)}(y-q_*)+\xi''_0(q_0)^2\beta^2(y-q_*)\left(-c+C\frac{1}{\beta^{1/2-\delta}}\right)
\leq (y-q_*)\beta^2\left( \frac{\xi'''_0(1)}{\xi''_0(q_0)\beta^2}-c^\prime+C^\prime\frac{1}{\beta^{1/2-\delta}}\right) 
\]
where $c'$ and $C'$ are independent of $\beta$. Since the second term in the last inequality is $-c'+o_\beta(1)$, the result follows for $\beta$ sufficiently large. 

We now turn to the case $\xi_0''(0)>0$.  The reader will observe
that in the above, the lower bound $h\geq h_0$ was required only to
produce the lower bound $q_*\geq q_0>0$. Recall that by Lemma \ref{lem:lbq},  we have such a lowerbound for $\beta$ sufficiently large. For example, $\beta \geq (1+\eps)/\sqrt{\xi''_0(0)}$, for $\epsilon\in (0,1)$, yields $q_* \geq \eps/(1+\eps)>0$. If one
then adjusts the proofs of Proposition \ref{prop:23-AT} and Lemma \ref{lem:longtime-main-lem} \emph{mutatis mutandis}, the above argument still holds.
\end{proof}

Note that by Girsanov's theorem (Corollary \ref{cor:girsanov}), if we let 
\begin{equation}\label{eq:Psi}
\Psi\left(x,y\right)=\left(4\text{sech}^{3}\left(y\right)-6\text{sech}^{5}\left(y\right)\right)\text{sech}\left(x\right)
\end{equation}
we get that for $t> q_*$
\begin{equation}\label{eq:afterGirsanov}
\E_h\left(4 \sech^4(X_t)-6\sech^6(X_t)\right) = \chi(t;\beta,h)e^{-\frac{1}{2}(\xi'(t)-\xi'(q_*))}
\end{equation}
where 
\[
\chi(t;\beta,h) = \E\Psi\left(\mathbf{h}+\sqrt{{S}\left(t,\beta,h\right)}\mathbf{Z}\right),
\]
and
\begin{align}\label{eq:cov-and-h}
{S}  &=\beta^2\sigma(q_*)\left(\begin{matrix}1 & 1\\
1 & \sigma(t)/\sigma(q_{*})
\end{matrix}\right) \\
\mathbf{h}  &=h\left(1,1\right)
\end{align}
and $\mathbf{Z}$ is a standard Gaussian vector in $\R^{2}$. Here we define 
\[
\sigma(s)=\xi'_0(s)
\]
 for ease of notation.
Thus the problem is of the form studied in 
Sect.\ \ref{sec:disp-gauss}. 

We end this section with the following motivational proposition
 which follows using the techniques from 
Sect.\ \ref{sec:1d-disp} under the additonal assumption that hypothesis H holds (see p.\ \pageref{par:hypH}). The purpose of this proposition is to 
illustrate to the reader why they might expect Lemma \ref{lem:longtime-main-lem}. We remind the reader here that hypothesis H is not used in the proofs of the main results.
\begin{proposition} Assume hypothesis H holds. Let $\{(\beta_n,h_n)\}$ be a sequence belonging to the level set $\{(\beta,h) : \alpha(\beta,h) = \tilde{\alpha}\}$ such that $\beta_n\to\infty$ and $h_n\geq h_0>0$. Then,

\[
\lim_{n\to\infty}\xi''(q_*)\E_{h}\left[4\text{sech}^{4}\left(X_{q_{*}}\right)-6\text{sech}^{6}\left(X_{q_{*}}\right)\right]
=-\frac{4}{5}\tilde{\alpha}.
\]
\end{proposition}
\begin{proof}
Note that
\[
\xi''(q_*)\E_{h}\left[4\text{sech}^{4}\left(X_{q_{*}}\right)-6\text{sech}^{6}\left(X_{q_{*}}\right)\right]=4\tilde{\alpha} \left(1-\frac{3}{2}\frac{\E_{h}\left[\text{sech}^{6}\left(X_{q_{*}}\right)\right]}{\E_{h}\left[\text{sech}^{4}\left(X_{q_{*}}\right)\right]}\right).
\]
We then apply Corollary \ref{cor:limit-ratio} to find that
\[
\lim_{\beta\to\infty}\left(1-\frac{3}{2}\frac{\E_{h}\left[\text{sech}^{6}\left(X_{q_{*}}\right)\right]}{\E_{h}\left[\text{sech}^{4}\left(X_{q_{*}}\right)\right]}\right)=1-\frac{3}{2}\frac{\int\text{sech}^{6}\left(x\right)}{\int\text{sech}^{4}\left(x\right)}=1-\frac{3}{2}\cdot\frac{16/15}{4/3}=-\frac{1}{5}.
\]
Therefore,
\[
\lim\,\xi''(q_{*})\E_{h}\left[4\text{sech}^{4}\left(X_{q_{*}}\right)-6\text{sech}^{6}\left(X_{q_{*}}\right)\right]=-\frac{4}{5}\tilde{\alpha}.
\qquad
\]
\end{proof}
 
\subsection{Setting up the main estimate}\label{sec:setup}
\subsubsection{Translating to the language of 2D dispersive estimates}
We now translate to the setting of Sect. \ref{sec:2d-disp}. To 
this end, fix $\xi_0,h_0,$ and $\alpha_0$ as in the statement of 
Lemma \ref{lem:longtime-main-lem}. \textbf{For the remainder of this paper we think of these variables as fixed unless otherwise specified.}
Recall the definitions of the functions $q_*(\beta,h)$ and $\alpha(\beta,h)$ from \eqref{eq:q-alpha-def}. 
For each $\tau\in(0,1]$ define $t = t(\beta,h;\tau) \in (q_*,1]$ through the bijective relation
\begin{equation}\label{eq:taudefn}
\tau = \frac{\sigma(t)-\sigma(q_{*})}{\sigma(1)-\sigma(q_{*})}.
\end{equation}

The index set that we use will essentially be the level sets of $\alpha(\beta,h)$. To be precise, for each 
$\tilde{\alpha}\in(\alpha_0,1],\tau\in(0,1]$, define the (marked) level set 
\begin{equation}
T_{\tilde{\alpha},\tau} = \{ (\beta,h,\tau')\in (0,\infty)\times[h_0,\infty)\times\{\tau\} : \alpha(\beta,h) = \tilde{\alpha} \}.
\end{equation}
This will be our index set. 

Recall $S$ and $\mathbf{h}$ from (\ref{eq:cov-and-h}). Define $\Sigma(\beta,h;\tau)={S}(t(\beta,h;\tau),\beta,h)$ and $\mathbf{m}(\beta,h)=\mathbf{h}$. Thus for each $\tilde{\alpha}\in(\alpha_0,1], \tau \in (0,1]$, we have the $T_{\tilde{\alpha},\tau}$-indexed family $(\Sigma(\beta,h;\tau),\mathbf{m}(\beta,h;\tau))$.

Note that if we define $a=a(\beta,h;\tau)$ by
\begin{equation}\label{eq:a}
a(\beta,h;\tau)=\frac{\sigma(t(\beta,h;\tau))}{\sigma(q_{*}(\beta,h))},
\end{equation}
then $\Sigma$ is a multiple of the matrix
\[
\left(\begin{array}{cc}
1 & 1\\
1 & a
\end{array}\right)
\]
which has eigenvalues and (unnormalized) eigenvectors
\begin{align}\label{eq:eigenvalues}
\tilde{\lambda}_{1}  =\frac{1}{2}\left(2+\left(a-1\right)-\sqrt{\left(a-1\right)^{2}+4}\right) \qquad
\tilde{\lambda}_{2}  =\frac{1}{2}\left(2+\left(a-1\right)+\sqrt{\left(a-1\right)^{2}+4}\right)
\end{align}
and
\begin{align}\label{eq:eigenvectors}
\tilde{v}_{1}  =\left(\frac{1}{2}\left(-\left(a-1\right)-\sqrt{\left(a-1\right)^{2}+4}\right),1\right)\qquad
\tilde{v}_{2}  =\left(\frac{1}{2}\left(-\left(a-1\right)+\sqrt{\left(a-1\right)^{2}+4}\right),1\right).
\end{align}
Let $v_{1}=\tilde{v}_{1}/\norm{\tilde{v}_{1}}$ and $v_{2}=\tilde{v}_{2}/\norm{\tilde{v}_{2}}$
and let $\lambda_{i}=\beta^2\sigma(q_{*})\tilde{\lambda}_{i}$ . 
Note that these depend on $\beta,h,$ and $\tau$.
Let $\mathbf{m}=\mathbf{h}$. We remind the reader that
\begin{equation}
m_i = \gibbs{\mathbf{m},v_i}\quad i=1,2.
\end{equation}

We let
\[
\tilde{w}_1=(-1,1)
\text{ and }
\tilde{w}_2=(1,1)
\]
and let $w_i=\tilde{w_i}/\norm{\tilde{w_i}}$ . Define $\nu=\nu(\tilde{\alpha},\tau)$ by
\begin{equation}\label{eq:nu}
\nu(\tilde{\alpha},\tau)=\frac{3}{4}\tilde{\alpha}\cdot\tau.
\end{equation} 

Now we are in the setting of Sect. \ref{sec:2d-disp}. We now ask that the reader match the notation from that section and recall the definitions therein. We will use said notation from now on. As we will soon show, $\Psi$ 
is in the regularity class required for Theorem \ref{thm:2d-disp-bound}. Before we apply this theorem, however, it will be useful to 
control the related spectral variables.


\subsubsection{Asymptotic spectral theory for certain operators}

For $(a,b,\tilde{q},\theta)\in[0,1]\times\R_+\times[0,1]^2$ define
\begin{align*}
C_{0}(a,b,\tilde{q})&=\frac{3}{2}a+\Lambda_0\frac{\xi''_{0}(1)}{(\xi'_{0}(\tilde{q}))^{3/2}}\frac{1}{b}\\
C_{1}(a,b,\tilde{q};\theta)&=\frac{\theta}{2}\frac{\sigma'(1)}{\sigma(\tilde{q})\sigma'(\tilde{q})} C_0(a,b,\tilde{q})\\
C_{2}(a,b,\tilde{q};\theta)&=  \frac{\theta\Lambda_0}{2}\frac{\sigma'(1)}{(\sigma(\tilde{q}))^{3/2}}
 +\frac{1}{2b}\left( \theta\frac{\sigma''(1)}{\sigma'(\tilde{q})}C_{0}(a,b,\tilde{q}) + \sigma(1)C_{1}(a,b,\tilde{q};\theta)^2\right).
 \end{align*}
Note these are increasing functions of $a$ and $\theta$ and decreasing functions of $b$ and $\tilde{q}$. We will be thinking of $a, b,\tilde{q},$ and $\theta$ as $\alpha,\beta,q_*$, and $\tau$ respectively. We have the following estimates whose proofs are deferred to the appendix (Sect.\ \ref{sec:spectral}).
\begin{lemma}\label{lem:spec-lem}
For all $(\beta,h,\tau)\in T_{\tilde{\alpha},\tau}$ we have the estimates:
\begin{align}
\abs{\left\langle w_1,v_{2}\right\rangle }=\abs{\left\langle w_2,v_{1}\right\rangle } &\leq \frac{C_{1}(1,\beta,q_0;1)}{\sqrt{2}\beta^{2}}\left(1+\frac{C_{1}(1,\beta,q_0;1)}{2\beta^{2}}\right)  \\
\abs{\lambda_{1}-\nu(\tilde{\alpha},\tau)}&\leq\frac{1}{\beta}C_2(1,\beta,q_0;1)\\
\abs{\frac{1}{\lambda_{2}^{1/2}}}&\leq\frac{1}{\beta\sqrt{2\sigma(q_0)}}\\
\abs{\tilde{\lambda}_{2}-2} & \leq\frac{1}{\beta^{2}}C_{1}(1,\beta_{0},q_0;1)\left(1+\frac{C_{1}(1,\beta_{0},q_0;1)}{2\beta^{2}}\right)
\end{align}
\end{lemma}

We note the following motivational proposition, which is a consequence of these estimates. In particular, this result clarifies the asymptotic structure of the problem and explains the choices made above. Recall hypothesis H from p.\ \pageref{par:hypH}.
\begin{proposition} 
Assume hypothesis H holds. Let $\{(\beta_n,h_n)\}$ be a sequence belonging to the level set $\{(\beta,h) : \alpha(\beta,h) = \tilde{\alpha}\}$ with $\tilde{\alpha}\in(0,1]$, such that $\beta_n\to\infty$ and $h_n\geq h_0>0$.
Then for each $\tau\in(0,1]$, 
\[
\lim\Sigma^{-1/2} =\frac{1}{\sqrt{\frac{3}{4}\tilde{\alpha}\tau}}w_1\tensor w_1 = \nu^{-1/2}w_1\tensor w_1.
\]
\end{proposition}
\begin{proof}
From the previous definitions and the lemma above, we see that
\[
\lambda_{1}=\frac{3}{4}\alpha\tau +O(1/\beta),
\]
$\lambda_2\rightarrow\infty$, and $\{v_1,v_2\}\rightarrow\{w,v\}$ as $\beta\to \infty$.
\end{proof}


Define
\[
\Theta(a,b,\tilde{q};\theta) =\sqrt{C_2(1,b, \tilde{q};1)}
  +\frac{C_{1}(1,b, \tilde{q};1)}{b^{3/2}}\left(1+\frac{C_{1}(1,b, \tilde{q};1)}{2b^{2}}\right) \left(\sqrt{\frac{3}{4}a\theta}+1\right),
\]
and note it is increasing in $a$ and $\theta$, and decreasing in $b$ and $\tilde{q}$.
Recall the definition of $A$ from \eqref{eq:Aop-def}.

\begin{lemma}\label{lem:Aop-control}
For all $(\beta,h,\tau)\in T_{\tilde{\alpha},\tau}$ we have the bound
\[
\norm{A(\beta,h;\tau) - A(\infty)}\leq \frac{\Theta(1,\beta,q_0;1)}{\sqrt{\beta}}.
\]
\end{lemma}
\begin{proof}
In the following $\lambda_i$, $v_i$ are their values at $(\beta,h,\tau)\in T_{\tilde{\alpha},\tau}$. 
Observe that
\begin{align*}
\norm{A(\beta,h;\tau)-A(\infty)} & \leq \norm{ \lambda_1^{1/2}v_1\tensor e_1 - \nu^{1/2}w_1\tensor e_1}+\norm{v_2\tensor e_2-w_2\tensor e_2}\\
&= \norm{ \lambda_1^{1/2}v_1- \nu^{1/2}w_1}_2+\norm{v_2-w_2}_2 \\
&\leq \abs{\lambda_1^{1/2}-\nu^{1/2}}+\nu^{1/2}\norm{v_1-w_1}+\norm{v_2-w_2}\\
&\leq \sqrt{\abs{\lambda_1-\nu}} + \sqrt{2}(\nu^{1/2}+1)\abs{\innerp{v_2,w_1}},
\end{align*}
where in the last line we used that $\{v_1,v_2\}$ is an orthonormal basis. Combining this with Lemma \ref{lem:spec-lem} gives the result.
\end{proof}



Let $\Lambda_1 = \frac{\pi^2-6}{18\sqrt{2\pi}}=\frac{1}{2\sqrt{2\pi}}\int sech^4(x)x^2dx$. Then, we have the following lemma whose proof is an application of Lemma \ref{lem:1d-disp} (to the function $f(x)=\sech^4(x)$). 
\begin{lemma}\label{lem:h-lim}
For all $(\beta,h,\tau)\in T_{\tilde{\alpha},\tau}$, we have the inequality
\[
\abs{\tilde{\alpha}-\frac{4}{3}\frac{\xi''(q_{*})}{\sqrt{\xi'(q_{*})}}\frac{e^{-\frac{1}{2}\frac{h^{2}}{\xi'(q_{*})}}}{\sqrt{2\pi}}}\leq\Lambda_1\frac{\xi''\left(q_{*}\right)}{\left(\xi'\left(q_{*}\right)\right)^{3/2}}.
\]
Therefore, if $\beta$ is such that
\begin{equation}\label{eq:alpha-43-inequality}
\beta>\beta''(\xi_0,\alpha_0,h_0):=\Lambda_1\frac{\xi''_{0}\left(1\right)}{\alpha_0\left(\xi'_{0}\left(q_0\right)\right)^{3/2}}>0,
\end{equation}
then we have the bound
\[
\frac{h^{2}}{\beta^{2}}\leq2\xi'_{0}(1)\left(\log\beta+\Theta_{1}(\alpha_0,\beta,q_0)\right)
\]
where
\[
\Theta_{1}\left(a,b,\tilde{q}\right)=\log\left[\frac{4}{3\sqrt{2\pi}}\frac{\xi''_{0}\left(1\right)}{\sqrt{\xi'_{0}\left(\tilde{q}\right)}}\left(a-\Lambda_1\frac{\xi''_{0}\left(1\right)}{b\left(\xi'_{0}\left(\tilde{q}\right)\right)^{3/2}}\right)^{-1}\right].
\]
and $q_0$ is as per \eqref{eq:q_0}.
\end{lemma}

Note that $\Theta_1$ is decreasing in $a$, $b$, and $\tilde{q}$.

\begin{lemma}\label{lem:m1-control}
For all $(\beta,h,\tau)\in T_{\tilde{\alpha},\tau}$ with $\beta>\beta''$ \emph{(}see \eqref{eq:alpha-43-inequality}\emph{)}, we have that
\[
\abs{m_{1}} \leq \sqrt{2\xi'_{0}(1)\left(\log(\beta)+\Theta_{1}(\alpha_0,\beta,q_0)\right)}\frac{C_{1}(1,\beta, q_0;1)}{\beta}\left(1+\frac{C_{1}(1,\beta,q_0;1)}{2\beta^{2}}\right).
\]
\end{lemma}
\begin{proof}
Recall that 
\[
 m_{1}=\left\langle v_{1},\mathbf{m}\right\rangle =\sqrt{2}h\left\langle v_{1},w_2\right\rangle.
\]
Combining Lemma \ref{lem:spec-lem} and Lemma \ref{lem:h-lim} gives the result.
\end{proof}



\subsection{Main estimate}
\begin{lemma}\label{lem:2ddispappl}
There is a universal constant $C=C(\Psi)$, and a choice of $\beta'$ and a constant $K_1$ depending on $\alpha_0,\xi_0,$ and $h_0$  
such that for all $(\beta,h,\tau)\in T_{\tilde{\alpha},\tau}$ with $\beta\geq \beta'$, we have that
\begin{align*}
\abs{&\E\Psi\left(\mathbf{h}+\sqrt{\Sigma}\mathbf{Z}\right)-\frac{e^{-\frac{1}{2}\frac{m_{2}^{2}}{\lambda_{2}}}}{\lambda_2^{1/2}\sqrt{2\pi}}\E\innerp{\Psi}(\nu(\tilde{\alpha},\tau)^{1/2} z)} 
\leq\frac{1}{\lambda_2^{1/2}}\left[\frac{C}{2\sqrt{2\pi}\lambda_2}
+\left(\frac{1}{\pi}\frac{1}{\lambda_{2}^{1/2}}
+\frac{e^{-\frac{1}{2}\frac{m_{2}^{2}}{\lambda_{2}}}}{\sqrt{2\pi}}\right)K_1\frac{\log(\beta)^3}{\sqrt{\beta}}\right].
\end{align*}
\end{lemma}
\begin{proof}
This will follow from an application of Theorem \ref{thm:2d-disp-bound}. We begin by observing that
$\Psi \in \cA$ and satisfies the bounds
\[
\abs{\Psi(x,y)}\leq Ke^{-\norm{(x,y)}_2} \qquad \text{ and } \qquad Lip(\Psi)\leq K.
\]
Then, by the arguments in Sect. \ref{sec:setup}, we can apply Theorem \ref{thm:2d-disp-bound} to conclude that
\[
\abs{\E\Psi\left(\mathbf{h}+\sqrt{\Sigma}\mathbf{Z}\right)-\frac{e^{-\frac{1}{2}\frac{m_{2}^{2}}{\lambda_{2}}}}{\lambda_2^{1/2}\sqrt{2\pi}}\E\innerp{\Psi}(\nu(\tilde{\alpha},\tau)^{1/2} Z)} \leq
\frac{1}{\lambda^{1/2}_2}\left[\frac{1}{\lambda_{2}}\frac{1}{2\sqrt{2\pi}}\norm{\fG_\infty[\Psi]}_{L_{1}(\abs{y_{2}}^{2}dy_{2})}+\frac{e^{-\frac{1}{2}\frac{m_{2}^{2}}{\lambda_{2}}}}{\sqrt{2\pi}}\Delta_1+\frac{1}{\pi}\frac{1}{\lambda_{2}^{1/2}}\Delta_2\right],
\]
where for any $M\geq 2$, we have the estimates
\[
\Delta_{i}\leq K\left[\norm{A(\beta,h,\tau)-A(\infty)}\left(M^{3}(1+\frac{4}{M}\frac{1-e^{-M^{2}/2}}{\sqrt{2\pi}})\right)+M^2\abs{m_{1}}\right]+4c_{1}c\left[ e^{-\frac{M}{c}}(M+1)+ e^{-\frac{M^{2}}{2}}\right].
\]
Since $\tilde{\alpha}>\alpha_0$, we can take $\beta$ sufficiently large such that \eqref{eq:alpha-43-inequality} is satisfied. We can then apply Lemma \ref{lem:spec-lem} to control the spectral parameters. After choosing $\beta\geq e\vee\beta''$ and  $M=2\log\beta\geq 2$, we then have, by Lemmas \ref{lem:Aop-control} and \ref{lem:m1-control}, that
\begin{align*}
\Delta_i &\leq K \left(
\frac{(\log\beta)^{3}}{\sqrt{\beta}}\Theta(1,\beta,q_0;1)
+ \frac{(\log\beta)^2}{\beta}\sqrt{2\xi'_{0}(1)\left(\log(\beta)+\Theta_{1}(\alpha_0,\beta,q_0)\right)}\frac{C_{1}(1,\beta,q_0;1)}{\sqrt{2}}\left(1+\frac{C_{1}(1,\beta,q_0;1)}{2\beta^{2}}\right) \right.\\
&\qquad\qquad\left.+\frac{\log\beta}{\beta} \right).
\end{align*}
Since $\Theta_1$ has at most logarithmic growth in $\beta$, we see that the second  and third term together 
are $O(log(\beta)^3/\beta)$. Thus we see that for $\beta$ sufficiently large, there is a $K'$ such that 
\[
\Delta_i \leq K' \frac{(\log\beta)^3}{\sqrt{\beta}}
\]
Finally, by the last part of Theorem \ref{thm:2d-disp-bound}, we have
that 
\[
\norm{\fG_\infty[\Psi]}_{L_{1}(\abs{y_{2}}^{2}dy_{2})}\leq C(\Psi)<\infty.
\]
The result then follows by plugging in.
\end{proof}

Finally, we note the following facts.
\begin{lemma}\label{lem:auxbd1}
There exist constants $C,\beta'''$ depending only on $\xi_0,h_0$ such that
for all $(\beta,h,\tau)\in T_{\tilde{\alpha},\tau}$ with $\beta\geq\beta'''$, 
\[
e^{-\frac{1}{2}(\xi'(t)-\xi'(q_*))}\geq C.
\]
\end{lemma}
\begin{proof}
This follows immediately from  Lemma \ref{lem:lbq}  and Corollary \ref{cor:q-lim}.
\end{proof}
\begin{lemma}\label{lem:auxbd2}
There exists constants $K,\beta''''$ depending on $\xi_0,h_0,\alpha_0$ such that
for all  $(\beta,h,\tau)\in T_{\tilde{\alpha},\tau}$,  with $\beta\geq\beta''''$, 
\[
 \frac{1}{K\beta^2} \leq \frac{e^{-\frac{m^2_2}{2\lambda_2}}}{\sqrt{2\pi\lambda_2}}\leq  \frac{K}{\beta^2}.
\]
\end{lemma}
\begin{proof}
In the following, $K$ will denote a positive constant depending
on at most the aforementioned parameters. By
the triangle inequality, we have that
\[
e^{-\abs{\frac{m_{2}^{2}}{2\lambda_{2}}-\frac{h^{2}}{2\xi'(q_{*})}}}e^{-\frac{h^{2}}{2\xi'(q_{*})}}\leq e^{-\frac{m_{2}^{2}}{2\lambda_{2}}}\leq e^{\abs{\frac{m_{2}^{2}}{2\lambda_{2}}-\frac{h^{2}}{2\xi'(q_{*})}}}e^{-\frac{h^{2}}{2\xi'(q_{*})}}.
\]
Combining the fact that 
\[
m_{1}^{2}+m_{2}^{2}=2h^{2}
\]
with the bounds from Lemmas  \ref{lem:spec-lem}, \ref{lem:h-lim}, and \ref{lem:m1-control}, we see that
\[
\abs{\frac{m_{2}^{2}}{2\lambda_{2}}-\frac{h^{2}}{2\xi'(q_{*})}}\leq\frac{K}{\beta}
\]
for large enough $\beta$. An application of Lemma \ref{lem:h-lim} proves
that 
\[
\frac{1}{K\beta}\leq e^{-\frac{h^{2}}{2\xi'(q_{*})}}\leq\frac{K}{\beta}.
\]
Combining with the above then yields the inequality
\[
\frac{1}{K\beta}\leq e^{-\frac{m_{2}^{2}}{2\lambda_{2}}}\leq\frac{K}{\beta}
\]
for large enough $\beta$. Similarly, it follows from Lemma \ref{lem:spec-lem}
that
\[
\frac{1}{K\beta}\leq\frac{1}{\sqrt{\lambda_{2}}}\leq\frac{K}{\beta}
\]
for large enough $\beta$. Combining these bounds proves the result.
\end{proof}

The proof of the next fact is deferred to the appendix (Sect.\ \ref{sec:integralsign}). Recall the definition of $\Psi$  in \eqref{eq:Psi}.
\begin{fact}\label{fact:sign-fact}
For all $x\in\R$, 
\[
\innerp{\Psi}(x)<0.
\]
\end{fact}

We can now prove Lemma \ref{lem:longtime-main-lem}, which we restate for the convenience of the reader.
\begin{lemma*}\textbf{\ref{lem:longtime-main-lem}}
For all $\alpha_0,h_0>0$, there exist constants $c,C,\beta_0>0$ depending only on $\xi_0,h_0,\alpha_0$
such that for all $\beta,h$ satisfying $\beta\geq\beta_0$,  $h\geq h_0$, and $\alpha\in (\alpha_0,1]$, we have that
\[
\E_h\left(4 \sech^4(X_t)-6\sech^6(X_t)\right)\leq -\frac{c}{\beta^2}+\frac{C\log(\beta)^3}{\beta^{5/2}},\quad t\geq q_*.
\]
\end{lemma*}
\begin{proof}
In the following, $K$ will denote a positive constant that depends at most on the aforementioned parameters
but may change between lines. 
Recall equation \eqref{eq:afterGirsanov}, which states that
\[
\E_h\left(4 \sech^4(X_t)-6\sech^6(X_t)\right)
= e^{-\frac{1}{2}(\xi'(t) - \xi'(q_*))} \cdot \E\Psi\left(\mathbf{h}+\sqrt{{S}\left(t,\beta,h\right)}\mathbf{Z}\right)
\]
for all $t\in(q_*,1]$.
By the bijective correspondence \eqref{eq:taudefn}, we know that there is a $\tau\in(0,1]$ such that
\[
{S}(t,\beta,h)=\Sigma(\beta,h;\tau).
\]
Thus by Lemmas \ref{lem:spec-lem}, \ref{lem:2ddispappl}, and \ref{lem:auxbd2},
\[
\E\Psi\left(\mathbf{m}+\sqrt{S\left(t,\beta,h\right)}\mathbf{Z}\right)
\leq \frac{e^{-\frac{m^2_2}{2\lambda_2}}}{\sqrt{2\pi\lambda_2}}\E\innerp{\Psi}(\nu^{1/2}(\alpha(\beta,h),\tau)Z)
+K\left(\frac{\log(\beta)^3}{\beta^{5/2}}\right)
\]
for sufficiently large $\beta$. 
By Fact \ref{fact:sign-fact}, there is a $c>0$ such that for all $s\in[0,1]$,
\[
\E\innerp{\Psi}(s Z)\leq -c.
\]
Hence, by Lemma \ref{lem:auxbd2} we have that
\[
\E\Psi\left(\mathbf{m}+\sqrt{S\left(t,\beta,h\right)}\mathbf{Z}\right)
\leq  \frac{-c}{K\beta^2}
+K\left(\frac{\log(\beta)^3}{\beta^{5/2}}\right)
\]
for sufficiently large $\beta$.
Finally, using Lemma \ref{lem:auxbd1} we can conclude the result for $t>q_*$. A continuity argument then gives the result at $t=q_*$.
\end{proof}
\section{Appendix}\label{sec:appendix}

The appendix is organized as follows. In Sect.\ \ref{sec:appendixPDE} we state some preliminary facts about the Parisi PDE. In Sect.\ \ref{sec:AC-SDE} we state some useful formulas regarding the Auffinger-Chen SDE. In Sect.\ \ref{sec:girsanov} we state a useful change of variables through Girsanov's theorem. In Sect.\ \ref{sec:spectral} we record some important spectral estimates to be used in Sect.\ \ref{sec:long-time}. In Sect.\ \ref{sec:integralsign} we bound a certain integral whose sign is of interest. In Sect.\ \ref{sec:elementaryRS} we give an elementary argument for RS at sufficiently high temperature and external field, which when combined with the main theorems proves boundedness of the exceptional set for our arguments. We end in Sect.\ \ref{sec:ATline-line} with a discussion regarding topological properties of the level sets of $\alpha$. 

\subsection{Well-posedness of the Parisi PDE} \label{sec:appendixPDE}

The following three propositions are taken from the authors' paper
\cite{JagTobSC15}. We call a continuous function $u:\left[0,1\right]\times\R\to\R$
with essentially bounded weak derivative $u_{x}$ a weak solution
of the Parisi PDE (\ref{eq:PPDE}) if it satisfies 
\[
0=\int_{0}^{1}\int_{\R}-u\phi_{t}+\frac{\xi''\left(t\right)}{2}\left(u\phi_{xx}+\mu\left[0,t\right]u_{x}^{2}\phi\right)\, dxdt+\int_{\R}\phi\left(1,x\right)\log\cosh x\, dx
\]
for every $\phi\in C_{c}^{\infty}\left((0,1]\times\R\right).$ 
\begin{proposition}
Let $\mu\in\Pr\left[0,1\right]$. There exists a unique weak solution
$u$ to the Parisi PDE. The weak solution $u$ to (\ref{eq:PPDE}) has higher regularity:
\begin{itemize}
\item $\partial_{x}^{j}u\in C_{b}\left(\left[0,1\right]\times\R\right)$
for $j\geq1$
\item $\partial_{t}\partial_{x}^{j}u\in L^{\infty}\left(\left[0,1\right]\times\R\right)$
for $j\geq0$. 
\end{itemize}
For all $j\geq1$, the derivative $\partial_{x}^{j}u$ is a weak solution
to
\[
\begin{cases}
\left(\partial_{x}^{j}u\right)_{t}+\frac{\xi''\left(t\right)}{2}\left(\left(\partial_{x}^{j}u\right)_{xx}+\mu\left[0,t\right]\partial_{x}^{j}u_{x}^{2}\right)=0 & \left(t,x\right)\in\left(0,1\right)\times\R\\
\partial_{x}^{j}u\left(1,x\right)=\frac{d^{j}}{dx^{j}}\log\cosh x & x\in\R
\end{cases}.
\]

\end{proposition}

\begin{proposition}
\label{prop:Lipmuu} Let $\mu,\tilde{\mu}\in\Pr[0,1]$ and $u,\tilde u$ be the corresponding solutions to 
the Parisi PDE. Then
\begin{align*}
 \norm{u-\tilde{u}}_{\infty}  \leq \xi''\left(1\right) d(\mu,\tilde{\mu}) \qquad\text{  and }\qquad
 \norm{u_x-\tilde{u}_x}_{\infty} \leq \exp\left(\xi'\left(1\right) \right)\xi''\left(1\right) d(\mu,\tilde{\mu}).
\end{align*}
\end{proposition}

\begin{proposition}
The solution $u$ to the Parisi PDE satisfies $\abs{u_{x}}<1$ and
$0<u_{xx}\leq 1$.
\end{proposition}

\subsection{The Auffinger-Chen SDE}\label{sec:AC-SDE}
Recall the Auffinger-Chen SDE from (\ref{eq:AC-SDE}),
\begin{align*}
dX_{t} & =\xi''\left(t\right)\mu[0,t]u_x\left(t,X_{t}\right)\, dt +\sqrt{\xi''\left(t\right)}\, dW_{t}\\
X_{0} & =h
\end{align*}
which has infinitesimal generator
\[
\cL_{t,\mu}=\frac{\xi''(t)}{2}(\Delta + 2\mu[0,t] u_x(t,x)\partial_x ).
\]
Note this has coefficients which are uniformly bounded in time and Lipschitz in space by Sect.\ \ref{sec:appendixPDE}. 
We now summarize some basic properties of the SDE which will be used in the subsequent. Their proofs
are standard applications of It\^o's lemma (see \cite{stroock1979multidimensional}) so they are omitted.  
\begin{lemma}
We have
\begin{align*}
\E_{h}\left[u_{x}^{2}\left(s,X_{s}\right)\right] & =  \int_{0}^{s}\xi''\left(t\right)\E_{h}\left[u_{xx}^{2}\left(t,X_{t}\right)\right]\, dt+u_{x}^{2}\left(0,h\right)\\
\E_{h}\left[u_{xx}^{2}\left(s,X_{s}\right)\right] & =  \int_{0}^{s}\xi''\left(t\right)\E_{h}\left[u_{xxx}^{2}\left(t,X_{t}\right)-2\mu\left[0,t\right]u_{xx}^{3}\left(t,X_{t}\right)\right]\, dt+u_{xx}^{2}\left(0,h\right)\\
\frac{d}{ds}\E_{h}\left[u_{x}^{2}\left(s,X_{s}\right)\right] & =\xi''\left(s\right)\E_{h}\left[u_{xx}^{2}\left(s,X_{s}\right)\right]\\
\frac{d}{ds}^{+}\E_{h}\left[u_{xx}^{2}\left(s,X_{s}\right)\right] & =\xi''\left(s\right)\E_{h}\left[u_{xxx}^{2}\left(s,X_{s}\right)-2\mu\left[0,s\right]u_{xx}^{3}\left(s,X_{s}\right)\right]\\
\frac{d}{ds}^{-}\E_{h}\left[u_{xx}^{2}\left(s,X_{s}\right)\right] & =\xi''\left(s\right)\E_{h}\left[u_{xxx}^{2}\left(s,X_{s}\right)-2\mu[0,s)u_{xx}^{3}\left(s,X_{s}\right)\right].
\end{align*}
\end{lemma}

\subsection{A change of measure formula} \label{sec:girsanov}
\begin{lemma}
Fix a measurable space $(\Omega,\cF)$. Let $Q$ be a probability measure such that $X_t$ solves (\ref{eq:AC-SDE}).
Then, there is a unique probability measure $P$ with
\[
R(t)= \frac{dQ}{dP}= \exp\left[ \int _0^{t} \mu[0,s]\, du\left( s,X_s\right)\right].
\]
Moreover, $X_t$ is distributed like $Y_t$ with respect to $P$, where $Y_t$ solves $dY_t = \sqrt{\xi''(t)}dW_t$ and $W_t$ is 
a standard Brownian motion with respect to $P$.
\end{lemma}
\begin{proof}
We apply Girsanov's theorem (see \cite[\lemmaname~6.4.]{stroock1979multidimensional}) directly. In particular, in the notation of the reference, if let
\[
c(t) = \mu[0,t]u_x(t,Y_t),
\]
$a(t)=\xi''(t)$, and $b=0$, we see that the Cameron-Martin-Girsanov exponential is of the form
\[
R(t)=\exp\left[\int_0^{t}\mu[0,t]u_x(t,Y_t)dY_t-\frac{1}{2}\int_0^{t} \xi''(t)\mu[0,t]^2u_x(t,Y_t)^2dt\right].
\]
Since $u$ solves the Parisi PDE, we see that its It\^o differential with respect to $dY_t$ is 
\[
du(t,Y_t)=-\frac{\xi''}{2}\mu u_x^2 dt + u_x dY_t.
\]
The result then follows by rearrangement.
\end{proof}
\begin{lemma}
We get the integration by parts formula: 
\[
\int_{0}^{t}\nu[0,s]du(t,Y_{t})=\int_{0}^{t}u(t,Y_{t})-u(s,Y_{s})d\nu(s)
\]
for $\nu$ a probability measure on $[0,1]$ .\end{lemma}
\begin{corollary}\label{cor:girsanov}
We have
\[
R(t)=e^{\int_{0}^{t}u(t,Y_{t})-u(q,Y_{q})d\mu(q)}.
\]
In particular if $\mu =\delta_q$, we have
\[
R(t)=\frac{\sech(Y_q)}{\sech(Y_t)}e^{- \frac{1}{2}(\xi'(t)-\xi'(q))},\quad t\geq q.
\]
\end{corollary}
For the reader more familiar with the work of \cite{TalBK11,TalBK11vol2}, we would like to demonstrate that this Girsanov argument also allows one to translate between that work, the work of Auffinger and Chen in \cite{AuffChen13,AuffChen14}, and the authors in \cite{JagTobSC15}. For example, in the notation of \cite[Chap. 13]{TalBK11vol2} the function $g$ defined in \eqref{eq:g-def}, $g'$, and the main family of integrals studied in Sect. \ref{sec:long-time} can also be written as follows. Let $z,z'$ be standard Gaussians, let $Y=h+\xi'(q_*)^{1/2}z$, $Y'=Y+(\xi(t)-\xi(q_*))^{1/2}z'$
and let $\E$ and $\E'$ denote the expectations with respect to $z$ and $z'$ respectively.
Then
\begin{equation}\label{eq:g-def-alt}
\begin{aligned}
g(y) &= \E\frac{\E'\left(\tanh^2(Y')\cosh(Y'))\right)}{\E_{z'}\cosh\left(Y'\right)}\\
g'(y) &= \xi''(t)\E\frac{\E'\sech^4(Y')\cosh(Y')}{\E'\cosh(Y')}\\
\E 4\sech^4(X_t)-6\sech^6(X_t)&=4\E \frac{\E'\sech^3(Y')}{\E'\cosh(Y')}-6\E\frac{\E'\sech^5(Y')}{\E'\cosh(Y')}
\end{aligned}
\end{equation}

In particular, the reader will observe that, judiciously applied, this Girsanov argument can be seen to relate the representation for these functions obtained through the dynamic programming principle and the Cole-Hopf formula.

\subsection{Asymptotic spectral theory for certain operators} \label{sec:spectral}
In this section we prove Lemma \ref{lem:spec-lem}. We begin with some preliminary estimates. Then the lemma 
is proved at the end of this section.  The definitions from Sect. \ref{sec:setup} will be 
used throughout this section.

We observe the following fact from calculus that will be used repeatedly in the subsequent. 

\begin{fact}\label{fact:calcfact}
We have that
\[
0\leq\sqrt{x+1}-1\leq\frac{1}{2}x,\quad x\geq0.
\]
\end{fact}

We will also use the following bound frequently: for $(\beta,h,\tau)\in T_{\tilde{\alpha},\tau}$,
\begin{align}\label{eq:abound}
\abs{\frac{a-1}{2}}&=\abs{\frac{\sigma(t)-\sigma(q_{*})}{2\sigma(q_{*})}}
=\frac{\tau}{2}\abs{\frac{\sigma(1)-\sigma(q_{*})}{\sigma(q_{*})}}
\leq\frac{\tau}{2}\frac{\sigma'(1)}{\sigma(q_*)\sigma'(q_*)}\sigma'(q_{*})\abs{1-q_{*}}
\leq\frac{1}{\beta^{2}}C_{1}(\tilde{\alpha},\beta,q_*;\tau).
\end{align}

The notation $\norm{\cdot}$ will refer to the $\ell_2$-norm throughout.

\subsubsection{Estimates on the eigenvectors}

Our goal will be to show 
\begin{lemma} \label{lem:eigenvectorbds}
For $(\beta,h,\tau)\in T_{\tilde{\alpha},\tau}$, we have that
\[
\abs{\left\langle w_1,v_{2}\right\rangle }=\abs{\left\langle w_2,v_{1}\right\rangle }\leq\frac{C_{1}(\tilde{\alpha},\beta,q_*;\tau)}{\sqrt{2}\beta^{2}}\left(1+\frac{C_{1}(\tilde{\alpha},\beta,q_*;\tau)}{2\beta^{2}}\right).
\]
\end{lemma}
\begin{proof}
Since $\norm{\tilde{v}_{1}}\geq1$,
\begin{align*}
\abs{\left\langle w_2,v_{1}\right\rangle } & \leq\frac{1}{\sqrt{2}}\abs{\left\langle \tilde{w_2},\tilde{v}_{1}\right\rangle }
 =\frac{1}{\sqrt{2}}\abs{1+\frac{1}{2}\left(-\left(a-1\right)-\sqrt{\left(a-1\right)^{2}+4}\right)}
  =\frac{1}{\sqrt{2}}\abs{(1-\sqrt{1+\frac{\left(a-1\right)^{2}}{4}})-\frac{a-1}{2}}\\
 &\leq\frac{1}{\sqrt{2}}\left(\abs{\frac{a-1}{2}}+\abs{1-\sqrt{1+\frac{\left(a-1\right)^{2}}{4}}}\right)
 \leq\frac{1}{\sqrt{2}}\left(\abs{\frac{a-1}{2}}+\frac{1}{2}\abs{\frac{a-1}{2}}^{2}\right)\\
 &= \frac{1}{\sqrt{2}}\abs{\frac{a-1}{2}}\left(1+\frac{1}{2}\abs{\frac{a-1}{2}}\right).
\end{align*}

By \eqref{eq:abound}, 
\begin{equation*}
\abs{\left\langle w_2,v_{1}\right\rangle }  \leq\frac{1}{\sqrt{2}}\frac{1}{\beta^{2}}C_{1}\left(1+\frac{1}{2\beta^{2}}C_{1}\right).
\end{equation*}
\end{proof}

\subsubsection{Estimates on the eigenvalues}

We will prove
\begin{lemma}\label{lem:eigenvaluebds-1}
For $(\beta,h,\tau)\in T_{\tilde{\alpha},\tau}$, we have that 
\begin{align*}
\abs{\lambda_{1}-\nu}\leq\frac{1}{\beta}C_2(\tilde{\alpha},\beta,q_*;\tau)\qquad\text{and}\qquad
\abs{\frac{1}{\lambda_{2}^{1/2}}}&\leq\frac{1}{\beta\sqrt{2\sigma(q_{*})}}.
\end{align*}
\end{lemma}
\begin{proof}
The second estimate follows from the fact that since $a\geq1$,
so that 
\begin{align*}
\tilde{\lambda}_{2} & =\frac{1}{2}\left(2+\left(a-1\right)+\sqrt{\left(a-1\right)^{2}+4}\right)
 \geq\frac{1}{2}\left(2+\sqrt{4}\right)=2.
\end{align*}

Now we prove the first estimate. By the triangle inequality,
\begin{align*}
\abs{\lambda_{1}-\nu} & \leq\beta^{2}\sigma(q_{*})\abs{\tilde{\lambda}_{1}-\frac{a-1}{2}}+\abs{\beta^{2}\sigma(q_{*})\frac{a-1}{2}-\frac{3}{4}\tilde{\alpha}\tau}=(i)+(ii).
\end{align*}
Since
\[
\tilde{\lambda}_{1}=\frac{1}{2}\left(2+\left(a-1\right)-\sqrt{\left(a-1\right)^{2}+4}\right) =\frac{\left(a-1\right)}{2}+\left(1-\sqrt{\frac{\left(a-1\right)^{2}}{4}+1}\right),
\]
we see that by Fact \ref{fact:calcfact},
\[
 \abs{\tilde{\lambda_1} - \frac{a-1}{2} } = \abs{1-\sqrt{1+\frac{(a-1)^2}{4}}} \leq \frac{1}{2}\abs{\frac{a-1}{2}}^2.
\]
Hence by \eqref{eq:abound} and the fact that $\sigma$ is non-decreasing,
\[
(i)  \leq\beta^{2}\sigma(q_{*})\frac{1}{2}\abs{\frac{a-1}{2}}^2 
    \leq\frac{\sigma(1)}{2\beta^{2}}C_{1}(\tilde{\alpha},\beta,q_*;\tau)^2.
\]

Now to study $(ii)$. For some $c(q_*)\in(q_*,1)$ we have that
\begin{align*}
(ii) & =\frac{\tau}{2}\abs{\beta^{2}(\sigma(1)-\sigma(q_{*}))-\frac{3}{2}\tilde{\alpha}}
  =\frac{\tau}{2}\abs{\beta^{2}\sigma'(c(q_{*}))(1-q_{*})-\frac{3}{2}\tilde{\alpha}}\\
 & \leq\frac{\tau}{2}\abs{\sigma'(c(q_{*}))-\sigma'(q_{*})}\beta^{2}(1-q_{*})+\frac{\tau}{2}\abs{\xi''(q_{*})\left(1-q_{*}\right)-\frac{3}{2}\tilde{\alpha}} =(iii)+(iv).
\end{align*}
We already know a bound on $(iv)$ by Corollary \ref{cor:q-lim}, so it remains to bound $(iii)$. 
Since $\sigma''$ is non-decreasing,
\[
\abs{\sigma'(c(q_{*}))-\sigma'(q_{*})}\leq\sigma''(1)(1-q_{*})\leq\frac{\sigma''(1)}{\sigma'(q_*)}\frac{C_{0}(\tilde{\alpha},\beta,q_*)}{\beta^{2}},
\]
where we have used Corollary \ref{cor:q-lim} in the last inequality. Thus, 
\[
(i)+(ii)\leq  \frac{1}{2\beta^{2}}\sigma(1)C_{1}^2 + \frac{1}{\beta}\frac{\tau}{2}\Lambda_0\frac{\sigma'(1)}{(\sigma(q_*))^{3/2}}+\frac{1}{\beta^{2}}\frac{\tau}{2}\frac{\sigma''(1)}{\sigma'(q_*)}C_{0} = \frac{C_2}{\beta}.
\]
\end{proof}
\begin{lemma}\label{lem:eigenvaluebds-2}
For all $(\beta,h,\tau)\in T_{\tilde{\alpha},\tau}$, we have that 
\[
\abs{\tilde{\lambda}_{2}-2}\leq \frac{C_1(\tilde{\alpha},\beta,q_*;\tau)}{\beta^2} \left (1+\frac{C_1(\tilde{\alpha},\beta,q_*;\tau)}{2\beta^2}\right).
\]
\end{lemma}
\begin{proof}
By Fact \ref{fact:calcfact} and \eqref{eq:abound},
\begin{align*}
\abs{\tilde{\lambda}_{2}-2} & =\abs{1+\left(\frac{a-1}{2}\right)+\sqrt{1+\left(\frac{a-1}{2}\right)^{2}}-2}
  \leq \abs{\frac{a-1}{2}} + \abs{\sqrt{1+\left (\frac{a-1}{2}\right)^2} -1}\\ 
&  \leq \abs{\frac{a-1}{2}}\left(1 + \frac{1}{2}\left (\frac{a-1}{2}\right)\right)
  \leq \frac{C_1}{\beta^2} \left ( 1+\frac{C_1}{2\beta^2}\right).
\end{align*}
\end{proof}
\subsubsection{Proof of Lemma \ref{lem:spec-lem}}

\begin{proof}\emph{of Lemma \ref{lem:spec-lem}}
Assembling the estimates in Lemmas \ref{lem:eigenvectorbds}--\ref{lem:eigenvaluebds-2}, we have for all $(\beta,h,\tau)\in T_{\tilde{\alpha},\tau}$ that
\begin{align*}
\abs{\left\langle w_1,v_{2}\right\rangle }&=\abs{\left\langle w_2,v_{1}\right\rangle }\leq\frac{C_{1}(\tilde{\alpha},\beta,q_*;\tau)}{\sqrt{2}\beta^{2}}\left(1+\frac{C_{1}(\tilde{\alpha},\beta,q_*;\tau)}{2\beta^{2}}\right)\\
\abs{\lambda_{1}-\nu}&\leq\frac{1}{\beta}C_2(\tilde{\alpha},\beta,q_*;\tau)\\
\abs{\frac{1}{\lambda_{2}^{1/2}}}&\leq\frac{1}{\beta\sqrt{2\sigma(q_{*})}}\\
\abs{\tilde{\lambda}_{2}-2}&\leq \frac{C_1(\tilde{\alpha},\beta,q_*;\tau)}{\beta^2} \left (1+\frac{C_1(\tilde{\alpha},\beta,q_*;\tau)}{2\beta^2}\right).
\end{align*}
Note that each $C_i(a,b,\tilde{q};\theta)$ is non-decreasing in $a$ and $\theta$ and non-increasing in $b$ 
and $\tilde{q}$. 

By definition of $q_0$ from \eqref{eq:q_0}, for $i=1,2$ we have that
\[
C_i(\tilde{\alpha},\beta,q_*;\tau) \leq C_i(1,\beta_0,q_0(\xi_0,\beta_0,h_0); 1).
\]
This implies the result. 
\end{proof}

\subsection{Proof of Fact \ref{fact:sign-fact}} \label{sec:integralsign}
\begin{lemma}
We have that
\[
\left\langle \Psi\right\rangle \left(x\right)=\int\left(4\text{sech}^{3}\left(\frac{x+y}{\sqrt{2}}\right)-6\text{sech}^{5}\left(\frac{x+y}{\sqrt{2}}\right)\right)\text{sech}\left(\frac{y-x}{\sqrt{2}}\right)\, dy<0
\]
for all $x\in\R$.\end{lemma}
\begin{proof}
A change of variables shows it is enough to prove that 
\[
g\left(x\right)=\frac{1}{\sqrt{2}}\left\langle \Psi\right\rangle \left(\sqrt{2}x\right)=\int\left(4\text{sech}^{3}y-6\text{sech}^{5}y\right)\text{sech}\left(y-2x\right)\, dy<0
\]
for all $x\in\R$. Note that
\[
\text{sech}\left(y-2x\right)=\frac{1}{\cosh\left(-2x\right)\cosh\left(y\right)+\sinh\left(-2x\right)\sinh\left(y\right)}=\frac{1}{a\cosh y+b\sinh y}
\]
where $a\left(x\right)=\cosh\left(-2x\right)$, $b\left(x\right)=\sinh\left(-2x\right)$.
Thus
\begin{align*}
g\left(x\right) & =\int_{-\infty}^{\infty}\frac{4\text{sech}^{3}y-6\text{sech}^{5}y}{a\cosh y+b\sinh y}\, dy
  =\int_{-\infty}^{\infty}\frac{\left(4\text{sech}^{2}y-6\text{sech}^{4}y\right)\text{sech}^{2}y}{a+b\tanh y}\, dy\\
 &=-\frac{2}{\sinh^{5}\left(2x\right)}\left(-3\sinh\left(4x\right)+4x\cosh\left(4x\right)+8x\right).
\end{align*}

Thus to show negativity of $g$ it suffices to show positivity
of 
\[
R\left(x\right)=\frac{-3\sinh\left(4x\right)+4x\cosh\left(4x\right)+8x}{\sinh^{5}\left(2x\right)}
\]
for all $x$. Note the denominator is negative for negative $x$ and
positive for positive $x$, so it suffices to show that the same is
true for 
\[
N\left(x\right)=-3\sinh\left(4x\right)+4x\cosh\left(4x\right)+8x,
\]
and to check that $R\left(0\right)>0$. Since $\frac{d^{j}}{dx^{j}}N\left(0\right)=0$
for $j=0,\dots,4$ and 
\[
\frac{d^{5}}{dx^{5}}N\left(x\right)=2048\cosh\left(4x\right)\left(1+2x\tanh\left(4x\right)\right)>0
\]
the result follows.\end{proof}

\subsection{The elementary argument for RS} \label{sec:elementaryRS}

In this section we prove the proposition:

\begin{proposition}\label{prop:elemenRS} For all models $\xi_0$ and $\beta_0>0$, there is an $h_0\left(\beta_0,\xi_0\right)$ such that 
 \[
  [0,\beta_0]\times[h_0,\infty) \subset RS.
 \]
\end{proposition}

Recall the definitions of $q_*$ and $\alpha$ from \eqref{eq:q-alpha-def}.
We will need the following preliminary result.
\begin{lemma}\label{lem:alphasmall}
Fix a model $\xi_0$ and $\beta_0>0$. Then for all $\epsilon>0$, there is an $h_{0}\left(\epsilon,\beta_0,\xi_0\right)$
such that
\[
\alpha\leq\epsilon
\]
for all $\beta\leq\beta_{0}$ and $h\geq h_{0}$.
\end{lemma}
\begin{proof}
Note that since $\sech^{4}\leq1$ and since $\xi'_0$ and $\xi''_0$ are non-decreasing, 
we have that
\begin{align*}
\beta^{2}\xi_0''(q_*)\E\sech^{4}\left(\beta\sqrt{\xi_0'(q_{*})}Z+h\right)&\leq\beta_0^{2}\xi_0''(1)\left(P\left(\abs{Z}\geq\frac{\delta}{\beta\sqrt{\xi'_0(q_{*})}}\right)+\sech^{4}(h-\delta)\right)\\
&\leq\beta_0^{2}\xi_0''(1)\left(2e^{-\frac{\delta^{2}}{2\beta_{0}^{2}\xi_0'(1)}}+\sech^{4}(h-\delta)\right)
\end{align*}
for $0<\delta<h$. Taking $\delta=h/2$ and $h\to\infty$ proves the result.
\end{proof}

\begin{proof}\emph{of Proposition \ref{prop:elemenRS}}
By Lemma \ref{lem:beta-alr} we may assume that $\beta\geq \beta_{*}=\frac{1}{\sqrt{\xi''_0(1)}}$. Then as in Lemma \ref{lem:lbq}, we observe that
\begin{align*}
  q_* &\geq q_0\\
 1-q_* &\leq \frac{\sqrt{\alpha}}{\beta_{*}\sqrt{\xi_0''(q_0)}}
\end{align*}
where $q_0(h) = \frac{1}{2}\tanh^2(h)$. Recall by Lemma \ref{lem:strats}, that if $(\beta,h)\in AT$ it suffices to prove that $g'\leq0$ on $[q_*,1]$ to conclude RS. We 
observe as in the proof of Theorem \ref{thm:long-time} that
\begin{align*}
g'(y)\leq \frac{\xi''_0(1)}{\xi''(q_*)}(\alpha-1)+\frac{\xi'''_0(1)}{\xi''_0(q_*)}(1-q_*)+ C\beta_0^4\left(\xi''_0(1)\right)^2(1-q_*).
\end{align*}
Using Lemma \ref{lem:alphasmall} and the estimates on $q_*$ given above, we may take $h\to \infty$ to conclude the result.
\end{proof}

\subsection{Is the AT line a line?} \label{sec:ATline-line}
In this section we briefly discuss some questions regarding the nature of the quantities and sets defined in (\ref{eq:q-alpha-def}).

The first question along these lines is as follows. Fix a model $\xi_0$. 
\begin{question}
For what region in the plane $(\beta,h)$ is $Q_*$ a singleton?
\end{question}
This question, it turns out, is very difficult to answer. For the SK model, this question has been resolved by Guerra and Lata\l a \cite{Guer01,TalBK11vol2}, where they (separately) showed uniqueness everywhere except for the set $h=0,\beta\geq1$. 
For models other than SK, it is far more complicated. For example, numerical studies show that in general the solution
to this fixed point equation is not unique. These studies suggest that when $h$ is large or when $\beta\geq h-\delta$, the solution is unique.
As the reader will see, the condition $\alpha \leq 1+\eps$ comes up in the analysis of related questions, so one is led to ask if this is exactly the region in which the unicity fails. 

Another natural question is regarding the set $\alpha=1$.
\begin{question}
Is the AT line actually a line? That is, is the AT line a (topological or smooth) curve?
\end{question}
This is also a delicate question. A step toward studying this question is the following lemma. 
\begin{lemma}
For any model and point $(\beta,h)$ with $\beta,h>0$ and $\alpha\leq 1$, the map $(\beta,h)\mapsto (q_*,\alpha)$ is $C^1$.
\end{lemma}
We note here that the condition $h>0$ is in general necessary as $\alpha$ should be zero on the set $h=0$ when $\xi''(0)=0$.
\begin{proof}
For ease of notation let $f=\xi'_0$. Consider, as usual, the map $F:\R^2\times\R^2 \rightarrow \R^2$, 
defined by
\[
F(\beta,h;q,\alpha)=(\E \tanh^2(\beta\sqrt{f(q)}z+h)-q,f'(q)\E \sech^4(\beta\sqrt{f(q)}z+h)-\alpha).
\]
Note that this is $C^1$. We see that the differential in $(q,\alpha)$ is lower triangular
\[
\left(\begin{matrix} 
\partial_q F_1&\partial_\alpha F_1\\
\partial_q F_2 &\partial_\alpha F_2\\
\end{matrix}\right)
=  \left(  \begin{matrix} 
      a & 0 \\ 
      b & -1 \\
   \end{matrix}\right)
\]
so that it suffices to show that $\partial_q F_1$ is non-zero. To see this, note that
\begin{align*}
\frac{\partial}{\partial q}\E\tanh^{2}(\beta\sqrt{f(q)}z+h)-q & =\E2\tanh \sech^{2}(\beta\sqrt{f(q)}z+h)\beta\frac{f'(q)}{2\sqrt{f(q)}}z-1\\
 & =\beta^{2}f'(q)\E\left[2-\cosh\left(2X\right)\right]\sech^{4}(X)-1\\
 & =\alpha-1+\beta^{2}f'(q)\E\sech^{4}(X)\left(1-\cosh(2X)\right)<0
\end{align*}
provided $\alpha \leq 1+\delta$ for $\delta$ sufficiently small. The second line follows from an integration by parts.
Thus by the implicit function theorem the map from $(\beta,h)\mapsto(q_*,\alpha)$ is $C^1$. 
\end{proof}
This does not show that the 
set $\alpha=1$ is a curve, however it does show that for almost every $\alpha\in [0,1]$, the set $\alpha(\beta,h)=\alpha$ is a curve. In particular, it 
shows this for every $\alpha$ that is regular in the sense of Sard. To get the result precisely when $\alpha=1$ is a difficult calculus question. 


\bibliographystyle{plain}
\bibliography{parisimeasures}

\begin{thebibliography}{10}

\bibitem{ALR87}
Michael Aizenman, Joel~L. Lebowitz, and David Ruelle.
\newblock Some rigorous results on the {S}herrington-{K}irkpatrick spin glass
  model.
\newblock {\em Comm. Math. Phys.}, 112(1):3--20, 1987.

\bibitem{BABM05}
G{\'e}rard~Ben Arous, Leonid~V. Bogachev, and Stanislav~A. Molchanov.
\newblock {Limit theorems for sums of random exponentials}.
\newblock {\em Probability theory and related fields}, 132(4):579--612, 2005.

\bibitem{AuffChen14}
Antonio {Auffinger} and Wei-Kuo {Chen}.
\newblock {The Parisi formula has a unique minimizer}.
\newblock {\em ArXiv e-prints}, February 2014.

\bibitem{AuffChen13}
Auffinger {Auffinger} and Wei-Kuo {Chen}.
\newblock {On properties of Parisi measures}.
\newblock {\em Probability Theory and Related Fields}, to appear, March 2013.

\bibitem{Bov12}
Anton Bovier.
\newblock {\em Statistical Mechanics of Disordered Systems}.
\newblock Cambridge, 2012.

\bibitem{BovKurk04}
Anton Bovier and Irina Kurkova.
\newblock Derrida's generalised random energy models 1: models with finitely
  many hierarchies.
\newblock {\em Annales de l'institut Henri Poincar\'{e} (B) Probabilit\'{e}s et
  Statistiques}, 40(4):439--480, 2004.

\bibitem{Chen15}
Wei~Kuo Chen.
\newblock {Variational representations for the Parisi functional and the
  two-dimensional Guerra-Talagrand bound}.
\newblock {\em ArXiv e-prints}, January 2015.

\bibitem{AT78}
J.~R.~L. de~Almeida and David~J. Thouless.
\newblock {Stability of the Sherrington-Kirkpatrick solution of a spin glass
  model}.
\newblock {\em Journal of Physics A: Mathematical and General}, 11(5):983,
  1978.

\bibitem{IntegralTable}
I.~S. Gradshteyn and I.~M. Ryzhik.
\newblock {\em Table of integrals, series, and products}.
\newblock Elsevier/Academic Press, Amsterdam, seventh edition, 2007.
\newblock Translated from the Russian, Translation edited and with a preface by
  Alan Jeffrey and Daniel Zwillinger, With one CD-ROM (Windows, Macintosh and
  UNIX).

\bibitem{Guer01}
Francesco Guerra.
\newblock Sum rules for the free energy in the mean field spin glass model.
\newblock In {\em Mathematical physics in mathematics and physics ({S}iena,
  2000)}, volume~30 of {\em Fields Inst. Commun.}, pages 161--170. Amer. Math.
  Soc., Providence, RI, 2001.

\bibitem{GuerTon02}
Francesco Guerra and Fabio~Lucio Toninelli.
\newblock Quadratic replica coupling in the {S}herrington-{K}irkpatrick mean
  field spin glass model.
\newblock {\em J. Math. Phys.}, 43(7):3704--3716, 2002.

\bibitem{JagTobSC15}
Aukosh Jagannath and Ian Tobasco.
\newblock {A dynamic programming approach to the Parisi variational problem}.
\newblock {\em ArXiv e-prints}, February 2015.

\bibitem{MPV87}
Marc M{\'e}zard, Giorgio Parisi, and Miguel~Angel Virasoro.
\newblock {\em Spin glass theory and beyond}, volume~9.
\newblock World scientific Singapore, 1987.

\bibitem{PanchConv05}
Dmitriy Panchenko.
\newblock A question about the {P}arisi functional.
\newblock {\em Electron. Commun. Probab.}, 10:no. 16, 155--166, 2005.

\bibitem{Panch05}
Dmitry Panchenko.
\newblock Free energy in the generalized {S}herrington-{K}irkpatrick mean field
  model.
\newblock {\em Rev. Math. Phys.}, 17(7):793--857, 2005.

\bibitem{Panch12}
Dmitry Panchenko.
\newblock The {S}herrington-{K}irkpatrick model: an overview.
\newblock {\em Journal of Statistical Physics}, 149(2):362--383, 2012.

\bibitem{PanchSKbook}
Dmitry Panchenko.
\newblock {\em The Sherrington-Kirkpatrick model}.
\newblock Springer, 2013.

\bibitem{PanchPF14}
Dmitry Panchenko.
\newblock The {P}arisi formula for mixed {$p$}-spin models.
\newblock {\em Ann. Probab.}, 42(3):946--958, 2014.

\bibitem{stroock1979multidimensional}
Daniel~W. Stroock and S.~R.~Srinivasa Varadhan.
\newblock {\em Multidimensional diffussion processes}, volume 233.
\newblock Springer Science \& Business Media, 1979.

\bibitem{Tal02}
Michel Talagrand.
\newblock On the high temperature phase of the {S}herrington-{K}irkpatrick
  model.
\newblock {\em Ann. Probab.}, 30(1):364--381, 2002.

\bibitem{TalSphPF06}
Michel Talagrand.
\newblock Free energy of the spherical mean field model.
\newblock {\em Probab. Theory Related Fields}, 134(3):339--382, 2006.

\bibitem{TalPM06}
Michel Talagrand.
\newblock Parisi measures.
\newblock {\em Journal of Functional Analysis}, 231(2):269 -- 286, 2006.

\bibitem{TalPF}
Michel Talagrand.
\newblock {The Parisi formula.}
\newblock {\em {Ann. Math. (2)}}, 163(1):221--263, 2006.

\bibitem{TalBK11}
Michel Talagrand.
\newblock {\em Mean field models for spin glasses. {V}olume {I}}, volume~54 of
  {\em Ergebnisse der Mathematik und ihrer Grenzgebiete. 3. Folge. A Series of
  Modern Surveys in Mathematics [Results in Mathematics and Related Areas. 3rd
  Series. A Series of Modern Surveys in Mathematics]}.
\newblock Springer-Verlag, Berlin, 2011.
\newblock Basic examples.

\bibitem{TalBK11vol2}
Michel Talagrand.
\newblock {\em Mean field models for spin glasses. {V}olume {II}}, volume~55 of
  {\em Ergebnisse der Mathematik und ihrer Grenzgebiete. 3. Folge. A Series of
  Modern Surveys in Mathematics [Results in Mathematics and Related Areas. 3rd
  Series. A Series of Modern Surveys in Mathematics]}.
\newblock Springer, Heidelberg, 2011.
\newblock Advanced replica-symmetry and low temperature.

\bibitem{Ton02}
Fabio Toninelli.
\newblock About the {A}lmeida-{T}houless transition line in the
  {S}herrington-{K}irkpatrick mean-field spin glass model.
\newblock {\em EPL (Europhysics Letters)}, 60(5):764, 2002.

\end{thebibliography}
\end{document}